\documentclass[12pt]{amsart}

\usepackage{amsmath}
\usepackage{amssymb}
\usepackage{amsthm}
\usepackage{bm}% bold math symbols
\usepackage{mathrsfs}
\usepackage{hyperref}
\usepackage{graphicx}
\usepackage{xcolor}
\usepackage{enumitem}
\setlist[enumerate]{parsep=0pt plus 4pt,topsep=0pt plus 4pt}
\usepackage{url}
\usepackage{mathtools}
\allowdisplaybreaks
\usepackage{epsfig}
\usepackage{xspace}
\newcommand\excise[1]{}

\newtheorem{thm}{Theorem}[section]
\newtheorem{lemma}[thm]{Lemma}
\newtheorem{prop}[thm]{Proposition}
\newtheorem{cor}[thm]{Corollary}

\newtheorem*{claim*}{Claim}
\theoremstyle{definition}
\newtheorem{defn}[thm]{Definition}
\newtheorem{remark}[thm]{Remark}
\newtheorem{example}[thm]{Example}

\newtheorem{hyp}[thm]{Hypotheses}

\numberwithin{equation}{section}

\newcounter{separated-sec}% to number Definition{d:collapse*} properly in intro
\newcounter{collapse-sec}% to number Definition{d:collapse*} properly in intro

\newcommand{\Ring}[1]{\ensuremath{\mathbb{#1}}}

\renewcommand\>{\rangle}
\newcommand\<{\langle}

\newcommand\II{\mathbb{I}}
\newcommand\JJ{\mathbb{J}}
\newcommand\KK{\mathcal{K}}
\newcommand\LL{\mathcal{L}}
\newcommand\MM{\mathcal{M}}

\newcommand\OO{\mathcal{O}}
\newcommand\RR{\Ring{R}}
\newcommand\TT{T}
\newcommand\XX{\mathcal{X}}

\newcommand\dd{\mathbf{d}}
\newcommand\ee{\mathbf{e}}
\newcommand\pp{\mathfrak{p}}

\newcommand\bP{\mathbf{P}}

\newcommand\cS{\mathcal{S}}

\newcommand\fc{C}% fluctuating cone (replaces \rC); why use roman font?

\newcommand\oC{{\hspace{.2ex}\overline{\hspace{-.2ex}C}\hspace{.2ex}}}

\newcommand\tM{\MM}

\newcommand\ve{\varepsilon}
\newcommand\vT{\smash{\makebox[0pt][l]{$T$}%
                \raisebox{1.4ex}{\scalebox{1.1}[.5]{$^{\rightarrow\!\!}$}}}}%

\newcommand\bmu{{\bar\mu}}
\newcommand\bnu{{\bar\nu}}
\newcommand\hmu{{\widehat\mu}}

\newcommand\Ci{C_\infty}

\newcommand\Ti{T_\infty}
\newcommand\Cmu{C_{\hspace{-.2ex}\mu}}
\newcommand\Emu{E_\mu}

\newcommand\Fmu{F}
\newcommand\Smu{S_\bmu\MM}
\newcommand\SpM{S_p\MM}
\newcommand\TZM{\vT_Z\MM}
\newcommand\TZR{\vT_Z R}
\newcommand\TZX{\vT_Z\XX}
\newcommand\Tmu{T_\bmu\MM}
\newcommand\TpM{T_p\MM}
\newcommand\bPi{\bP_\infty}

\newcommand\dis{\displaystyle}
\newcommand\muL{\mu^\LL}

\newcommand\Chmu{C_{\hspace{-.2ex}\hmu}}
\newcommand\Cmuz{C_{\hspace{-.2ex}\mu_Z}}
\newcommand\Ehmu{E_{\hspace{.1ex}\hmu}}
\newcommand\Emuz{E_{\mu_Z}}

\newcommand\Tmux{T_\bmu\XX}
\newcommand\bnui{\bnu_\infty}

\newcommand\oCmu{\oC_{\hspace{-.45ex}\mu}}

\newcommand\nablaO{\nabla_{\!\OO}}
\newcommand\nablaq{\nabla_{\hspace{-.25ex}q\hspace{.25ex}}}

\newcommand\nablabmu{\nabla_{\hspace{-.25ex}\bmu}}

\renewcommand\implies{\Rightarrow}

\newcommand\wh[1]{{\widehat{#1}}}

\newcommand\Cni[1]{C_{\nu_{#1}}}
\newcommand\interior[1]{{\kern0pt#1}^{\mathrm{o}}}

\DeclareMathOperator*\argmin{argmin}
\DeclareMathOperator\CAT{CAT}
\DeclareMathOperator\hull{hull}
\DeclareMathOperator\supp{supp}
\DeclareMathOperator\codim{codim}

\DeclarePairedDelimiter\abs{\lvert}{\rvert}
\DeclarePairedDelimiter\norm{\lVert}{\rVert}

\definecolor{teal}{HTML}{029386}
\begin{document}%%%%%%%%%%%%%%%%%%%%%%%%%%%%%%%%%%%%%%%%%%%%%%%%%%%%%%
%%%%%%%%%%%%%%%%%%%%%%%%%%%%%%%%%%%%%%%%%%%%%%%%%%%%%%%%%%%%%%%%%%%%%%
%%%%%%%%%%%%%%%%%%%%%%%%%%%%%%%%%%%%%%%%%%%%%%%%%%%%%%%%%%%%%%%%%%%%%%

\mbox{}\vspace{-4.5ex}
\title[Geometry of measures on smoothly stratified metric spaces]%
{Geometry of measures on\\smoothly stratified metric spaces}
\author{Jonathan C. Mattingly}
\address{\rm Departments of Mathematics and of Statistical Sciences,
Duke University, Durham, NC 27708}
\urladdr{\url{https://scholars.duke.edu/person/jonathan.mattingly}}
\author{Ezra Miller}
\address{\rm Departments of Mathematics and of Statistical Sciences,
Duke University, Durham, NC 27708}
\urladdr{\url{https://scholars.duke.edu/person/ezra.miller}}
\author[Do Tran]{Do Tran}%
% \\\\\vspace{-5ex}\smash{\raisebox{2ex}{(DRAFT---DO NOT DISTRIBUTE)}}}
\address{\rm Georg-August Universit\"at at G\"ottingen, Germany}
% Felix-Bernstein-Institute for Mathematical Statistics in the Biosciences
% \email{do.tranvan@uni-goettingen.de}

\makeatletter
  \@namedef{subjclassname@2020}{\textup{2020} Mathematics Subject Classification}
\makeatother
% Primary:
% 60D05 Geometric probability and stochastic geometry [See also 52A22, 53C65]
% 53C23 Global geometric and topological methods (à la Gromov);
%       differential geometric analysis on metric spaces
% 28C99 None of the above, but in "measures on spaces with additional structure"
% 57N80 Stratifications in topological manifolds
% 58A35 Stratified sets [See also 32S60]
% 62G20 Asymptotic properties of nonparametric inference
% 49J52 Nonsmooth analysis [See also 46G05, 58C50, 90C56]
% 62R20 Statistics on metric spaces
% 62R07 Statistical aspects of big data and data science 
% 53C80 Applications of global differential geometry to the sciences
% 58Z05 Applications of global analysis to the sciences
% 60B05 Probability measures on topological spaces
% 62R30 Statistics on manifolds
% 
% Secondary:
% 60F05 Central limit and other weak theorems
% 58K30 Global theory of singularities
% 57R57 Applications of global analysis to structures on manifolds
% 92B10 Taxonomy, cladistics, statistics in mathematical biology
% 
% \keywords{...}

\date{11 November 2023}

\begin{abstract}
Any measure~$\mu$ on a $\CAT(\kappa)$ space~$\MM$ that is stratified
as a finite union of manifolds and has local exponential maps near the
Fr\'echet mean~$\bmu$ yields a continuous \emph{tangential collapse}
$\LL: \Tmu \to \RR^m$ from the tangent cone of~$\MM$ at~$\bmu$ to a
vector space that preserves the Fr\'echet mean, restricts to an
isometry on the \emph{fluctuating cone} of directions in which the
Fr\'echet mean can vary under perturbation of~$\mu$, and preserves
angles between arbitrary and fluctuating tangent vectors at the
Fr\'echet mean.
\vspace{-4ex}
\end{abstract}

\maketitle
\tableofcontents
\vspace{-4.73ex}

%%%%%%%%%%%%%%%%%%%%%%%%%%%%%%%%%%%%%%%%%%%%%%%%%%%%%%%%%%%%%%%%%%%%%%
\section*{Introduction}\label{s:intro}%%%%%%%%%%%%%%%%%%%%%%%%%%%%%%%%
%%%%%%%%%%%%%%%%%%%%%%%%%%%%%%%%%%%%%%%%%%%%%%%%%%%%%%%%%%%%%%%%%%%%%%

\noindent
With the increasing recognition of singular spaces as sample spaces in
geometric statistics (see \cite{shadow-geom} and references listed
there), it becomes important to understand how the geometry of
singular spaces interacts with measures on them.  The most basic
questions in this context concern how to minimize
%enlargethispage{1.96ex}
\begin{equation}\label{eq:FrechetF}%\end{equation}
  F_\mu(p) = \frac12\int_\MM d(p,x)^2\,\mu(dx),
\end{equation}
the \emph{Fr\'echet function} of the given population measure~$\mu$ on
a singular space~$\MM$.  More generally, the intent is to carry out
statistical analysis with measures on singular spaces, including
estimators, summaries, confidence regions, and so on, based on
mathematical foundations like laws of large numbers and central limit
theorems.  This endeavor first requires the identification of spaces
with enough structure to carry through the relevant probability theory
while allowing maximal freedom regarding the types of singularities.
The goal here is therefore to identify a class of such spaces and
prove that their local properties around the Fr\'echet mean of a
measure~$\mu$ grant enough control to get a handle on variation of
Fr\'echet means upon perturbation of~$\mu$.

Spaces with curvature bounded above by~$\kappa$ in the sense of
Alexandrov, also called $\CAT(\kappa)$ spaces, always admit local
logarithm maps because they locally admit unique length-minimizing
geodesics (see \cite[Proposition~\textrm{II}.1.4]{bridson2013metric},
for example).  To suit the purposes of geometric probability, one
additional requirement in Definition~\ref{d:stratified-space} is that
this logarithm be locally invertible nearby~$\bmu$, which allows
information to be transferred from~$\MM$ to its tangent cone and back.
The other requirement, designed to control the singularities directly,
is that~$\MM$ be a~finite union of manifold strata, which seems
reasonable in statistical contexts, as the spaces involved are usually
semialgebraic
% (we know of no examples where they aren't)
(and thus can be triangulated \cite[\S II.2]{shiota97}) or are
explicitly described in terms of manifolds, such as from methods based
on manifold learning.
\enlargethispage{.75ex}%
(Manifold stratification is automatic for spaces with curvature
bounded below \cite{perelman1994}, but not above.)

The resulting \emph{smoothly stratified metric spaces} introduced in
Section~\ref{s:strat-spaces} provide
% ... evidence for the handle being quite strong ...
a strong handle on the geometry of their measures.  In smooth
settings, variation of Fr\'echet means in the limit takes place on the
tangent space \cite{bhattacharya-patrangenaru2003,
bhattacharya-patrangenaru2005}, which of course is a vector space.
The analogue for $\CAT(\kappa)$ spaces---the stratification and
exponential map are not required---is that the tangent cone is always
$\CAT(0)$
(see \cite[Section~1.2]{shadow-geom}
% \ref{b:tangent-cone}
for an exposition tailored to the current setting).  However, our
proofs of singular central limit theorems \cite{escape-vectors} reduce
further to the linear case, which necessitates the main result here
(Theorem~\ref{t:collapse}): the measure~$\mu$ has a continuous
\emph{tangential collapse} (Definition~\ref{d:collapse})
$$%$$ $
  \LL: \Tmu \to \RR^m
$$%$$ $
that preserves the Fr\'echet mean, restricts to an isometry on the
\emph{fluctuating cone~$\Cmu$} of directions in which the Fr\'echet
mean can vary under perturbation of~$\mu$
(Section~\ref{b:fluctuating-cone}), and preserves angles between
arbitrary and fluctuating tangent vectors.

The collapse~$\LL$ is constructed in Section~\ref{s:collapse} by
\emph{d\'evissage} (see especially Remarks~\ref{r:devissage},
\ref{r:resolved}, and~\ref{r:sticky}, along with
Definitions~\ref{d:step-devissage} and~\ref{d:dévissage}), a term from
algebraic geometry that means iteratively approaching a singular point
from less singular strata; e.g.,~see~\cite[Theorem~14.4 and its
proof]{Eis95}.  These tangential approaches are the \emph{limit
logarithm maps} from shadow geometry \cite{shadow-geom}.  To get
tangential collapse they are followed by convex geodesic projection
onto the relevant smooth stratum
(Definition~\ref{d:terminal-projection}).

The geometric arguments here require fairly weak hypotheses on the
measure~$\mu$ (Section~\ref{s:localized}): unique Fr\'echet mean and
no mass near its cut locus.  These conditions, with their resulting
convexity of the directional derivatives of the Fr\'echet function
(Section~\ref{b:conv_dd}), also suffice to prove a central limit
theorem for random tangent fields \cite{random-tangent-fields}, whose
convergence is the basis for more traditional central limit theorems
\cite{escape-vectors}, although additional analytic hypotheses are
required for that.

\subsection*{Acknowledgements}
DT was partially funded by DFG HU 1575/7.  JCM thanks the NSF RTG
grant DMS-2038056 for general support.
\vspace{-.5ex}

%%%%%%%%%%%%%%%%%%%%%%%%%%%%%%%%%%%%%%%%%%%%%%%%%%%%%%%%%%%%%%%%%%%%%%
\section{Geometric prerequisites}\label{s:prereqs}%%%%%%%%%%%%%%%%%%%%
%%%%%%%%%%%%%%%%%%%%%%%%%%%%%%%%%%%%%%%%%%%%%%%%%%%%%%%%%%%%%%%%%%%%%%

\noindent
Limit log maps on $\CAT(\kappa)$ spaces from the prequel
\cite{shadow-geom} form the basis for tangential collapse.  Only the
basics and cited prerequisites are reviewed here; see the prequel for
additional detail and relevant exposition.

\begin{defn}\label{d:angle}
The \emph{angle} between geodesics $\gamma_i: [0,\ve_i) \to \MM$ for
$i = 1,2$ emanating from~$p$ in a $\CAT(\kappa)$ space~$(\MM,\dd)$ and
parametrized by arclength is characterized by
$$%$$ $
  \cos\bigl(\angle(\gamma_1,\gamma_2)\bigr)
  =
  \lim_{t,s\to 0}\frac{s^2+t^2-\dd^2(\gamma_1(s),\gamma_2(t))}{2st}.
$$%$$ $
The geodesics~$\gamma_i$ are \emph{equivalent} if the angle between
them is~$0$.  The set $\SpM$ of equivalence classes is the \emph{space
of directions} at~$p$.
\end{defn}

\begin{lemma}\label{l:angular-metric}
The notion of angle makes the space~$\SpM$ of directions into a length
space whose \emph{angular metric}~$\dd_s$ satisfies
$$%$$ $
  \dd_s(V, W) = \angle(V, W)
  \text{ whenever } V, W \in \SpM
  \text{ with } \angle(V, W) < \pi
$$%$$ $
\end{lemma}
\begin{proof}
This is \cite[Proposition~1.7]{shadow-geom}, which in turn is
\cite[Lemma~9.1.39]{BBI01}.
\end{proof}

\begin{defn}\label{d:tangent-cone}
The \emph{tangent cone} at a point $p$ in a $\CAT(\kappa)$ space~$\MM$
is
$$%$$ $
  \TpM = \SpM \times [0,\infty) / \SpM \times \{0\},
$$%$$ $
whose \emph{apex} is often also called~$p$ (although it can be called
$\OO$ if necessary for clarity).  The \emph{length} of a vector $W =
W_p \times t$ with $W_p \in \SpM$ is $\|W\| = t$ in~$\TpM$.  The
\emph{unit tangent sphere}~$\SpM$ of length~$1$ vectors in~$\TpM$ is
identified with the space of~directions.
% combines \cite[Defs~1.9 and~1.10]{shadow-geom}
% with 1.10 being \ref{d:unit-tangent-sphere}
\end{defn}

\begin{defn}\label{d:inner-product}
Given tangent vectors $V, W \in \TpM$, their \emph{inner product} is
$$%$$ $
  \<V,W\>_p = \|V\|\|W\| \cos\bigl(\angle(V,W)\bigr).
$$%$$ $
The subscript $p$ is suppressed when the point~$p$ is clear from
context.
% \cite[Def~1.11]{shadow-geom}
\end{defn}

\begin{lemma}\label{l:inner-product-is-continuous}
For a fixed basepoint in a $\CAT(\kappa)$ space, the inner product
function $\<\,\cdot\,,\,\cdot\,\>_p: \TpM \times \TpM \to \RR$ is
continuous.
\end{lemma}
\begin{proof}
This is \cite[Lemma~1.21]{shadow-geom},
% it follows from the fact that the angular metric in
% Lemma~\ref{l:angular-metric} is a metric.
where it is derived from Lemma~\ref{l:angular-metric}.
\end{proof}

The angular metric $\dd_s$ induces a metric on the tangent cone
$\TpM$ which makes $\TpM$ into a length space.
% \cite[Def~1.12]{shadow-geom}
  
\begin{defn}\label{d:conical-metric}
The \emph{conical metric} on the tangent cone $\TpM$ of a
$\CAT(\kappa)$ space is
$$%$$ $
  \dd_p(V,W)
  =
  \sqrt{\|V\|^2 + \|W\|^2 - 2\<V,W\>}\
  \text{ for } V,W \in \TpM.
$$%$$ $
\end{defn}

\begin{lemma}[{\cite[Lemma~3.6.15]{BBI01}}]\label{l:flat-triangle}
Any geodesic triangle in $\TpM$ with one vertex at the apex is
isometric to a triangle in~$\RR^2$.
\end{lemma}

\begin{defn}\label{d:log-map}
Fix a point $p$ in a $\CAT(\kappa)$ space~$(\MM,\dd)$.  For each point
$v$ in the set $\MM' \subseteq \MM$ of points with a unique shortest
path to~$p$, write $\gamma_v$ for the unit-speed shortest path
from~$p$ to~$v$ and $V = \gamma_v'(0)$ for its tangent vector at~$p$.
Define the \emph{log~map}~by
\begin{align*}%\end{align*}%$$
  \log_p : \MM' & \to \TpM
\\
              v &\mapsto \dd(p,v) V.
\end{align*}%$$
% \comment{Def~1.14}
$\MM$ is \emph{conical} with \emph{apex}~$p$ if $\MM' = \MM$ and
$\log_p: \MM \to \TpM$ is an isometry.
\end{defn}

\begin{prop}[{\cite[Definition~3.1]{shadow-geom}}]\label{p:limit-tangent-cone}
Given a $\CAT(\kappa)$ space~$\MM$ and $\XX \!=\! \TpM$ with
apex~$\OO$, fix $Z \in T_\OO\XX$.  For points $q$ and~$q'$ in the
segment $\OO z$ between $\OO$ and $z =\nolinebreak \exp_\OO Z$, there
is a radial transport map $\pp_{q \to q'}$ that identifies $T_q\XX$
with $T_{q'}\XX$.  The \emph{limit tangent cone} along~$Z$ is the
% terminal object of the
direct limit
$$%$$ $
  \TZX = \varinjlim_{q \in \OO z} T_q\XX.
$$%$$ $
\end{prop}

\begin{defn}[{\cite[Definition~3.3]{shadow-geom}}]\label{d:limit-log}
Given a $\CAT(\kappa)$ space~$\MM$ whose tangent cone $\XX
=\nolinebreak \TpM$ has apex~$\OO$, fix $Z \in T_\OO\XX$.  The
\emph{limit log map} along $Z$~is
\begin{align*}%\end{align*}%$$
  \LL_Z: T_\OO\XX & \to \TZX
\\
             tV &\mapsto tV_Z \text{ for all } t \geq 0,
\end{align*}%$$
where $V_Z$ is the image in the limit tangent space~$\TZX$ of
$\pp_{\OO\to q} V$ for any $q \neq \OO$ in the geodesic segment
joining $\OO$ to~$z = \exp_\OO Z$.
\end{defn}

\begin{remark}\label{r:TvM-vs-TvX}
When $\XX$ is any $\CAT(0)$ cone with apex~$\OO$, such as $\XX =
\TpM$, the exponential map at~$\OO$ naturally identifies $\XX$
with~$T_\OO\XX$, so the limit log map~$\LL_Z$ naturally induces a map
$\XX \to \TZX$, also denoted~$\LL_Z$.  This convention is used often
in Section~\ref{b:mu-lim-log}.  To avoid cumbersome notation
like~$\vT_Z(\TpM)$, write $\TZM$ to denote~$\TZX$.  Let us restate
some results from \cite{shadow-geom} under this convention for use in
Section~\ref{b:mu-lim-log}.
\end{remark}

\begin{prop}\label{p:limit-log-iso-on-small-geo}
Let~$\MM$ be $\CAT(\kappa)$ and $\XX = \TpM$.  Suppose that $Z \in
\SpM$ and $\gamma: [0,1] \to \SpM$ is a geodesic of length $\alpha <
\pi$ such that there is at most one point $t_0\in [0,1]$ with
$\angle\bigl(Z,\gamma(t_0)\bigr) \geq \pi$.  Then $\LL_Z$
isometrically maps $\gamma$ onto its image.
\end{prop}
\begin{proof}
This is a restatement of \cite[Propositions~3.9
% \ref{p:almost-iso-of-singular-radial-transport}
and~3.14]{shadow-geom}.
% \ref{p:almost_iso-of_limit-log:v2}
\end{proof}

\begin{lemma}\label{l:limit-log-preserve-npc}
If~$\MM$ is $\CAT(\kappa)$ and $Z \in \TpM$ then $\TpM$ and~$\TZM$ are
$\CAT(0)$~spaces.
\end{lemma}
\begin{proof}
The tangent cone at any point of~$\MM$ is $\CAT(0)$
\cite[Proposition~1.18]{shadow-geom}.
% \ref{p:tangent-cone-is-NPC}
The claim follows because limit log produces a tangent cone by
Definition~\ref{d:limit-log}.
\end{proof}

\begin{prop}\label{p:limit-log-contracts}
If~$\MM$ is $\CAT(\kappa)$ and $\XX = \TpM$ has apex~$\OO$, then the
limit log map is a contraction: for $\XX = \TpM$ and $Z
\in\nolinebreak T_\OO\XX$ and any $V,W \in \XX$,
$$%$$ $
  \angle(\LL_Z V, \LL_Z W) \leq \angle(V,W),
$$%$$ $
and equality holds if\/~$W = Z$.
\end{prop}
\begin{proof}
This is \cite[Proposition~3.13]{shadow-geom} and, for the equality
case,
\cite[Corollary~3.12]{shadow-geom}.
% \ref{c:LL-preserves-angle-to-Z}
\end{proof}

\begin{lemma}\label{l:sum=pi}
Fix smoothly stratified $\MM$ and $Z \in \SpM$ pointing toward a
stratum~$R$ of~$\MM$.  Then $\TZR$ is a vector space, so $-\LL_Z(Z)\in
\TZR$, and for any $W \in \TZM$,
$$%$$ $
  \angle(\LL_Z Z, W) + \angle(W, -\LL_Z Z)
  =
  \angle(\LL_Z Z, -\LL_Z Z)
  =
  \pi.
$$%$$ $
\end{lemma}
\begin{proof}
This is \cite[Proposition~3.8]{shadow-geom}
% \ref{p:sum=pi}
in the smoothly stratified setting.
\end{proof}

The main results of \cite{shadow-geom} are of course useful here,
specifically in the proofs of
Corollaries~\ref{c:inj-of-limit-log-on-Cmu}
and~\ref{c:LL-preserves-Cmu}, having been designed specifially for
those purposes.

% The \emph{shadow} of a tangent vector $Z \in T_\OO\XX$ is the set
% $\IZ$ of nonzero vectors that form an angle of~$\pi$ with~$Z$:
% $$%$$ $
%   \IZ = \{V \in T_\OO\XX \mid \angle(V,Z) = \pi\}.
% $$%$$ $

\begin{thm}[{\cite[Theorem~3.16]{shadow-geom}}]\label{t:isometry-limit-log}
Fix a $\CAT(\kappa)$ space~$\MM$ with tangent cone $\XX = \TpM$ and a
vector $Z \in T_\OO\XX$.  If $\KK \subseteq \Tmu$ is a geodesically
convex subcone containing at most one ray that has angle~$\pi$
with~$Z$, then the restriction $\LL_Z|_\KK: \KK \to\nolinebreak
\LL_Z(\KK)$ to~$\KK$ of the limit log map along~$Z$ is an isometry
onto its image.
\end{thm}

\begin{cor}[{\cite[Corollary~3.19]{shadow-geom}}]\label{c:hull-preserved-by-limit-log}
Fix a $\CAT(\kappa)$ space~$\MM$ with tangent cone $\XX = \TpM$ and a
vector $Z \in T_\OO\XX$.  Taking limit log along any vector $Z \in
T_\OO\XX$ subcommutes with taking convex cones in the sense that for
any subset $\cS \subseteq \XX$,
$$%$$ $
  \LL_Z(\hull\cS) \subseteq \hull\LL_Z(\cS),
$$%$$ $
where the \emph{hull} of a subset
% of a $\CAT(0)$ cone
is the smallest geodesically convex cone containing the subset.
% $\hull(\cS)$ is the smallest geodesically convex cone in~$\XX$
% containing~$\cS$, and similarly for~$\LL_Z(\cS)$
\end{cor}

%%%%%%%%%%%%%%%%%%%%%%%%%%%%%%%%%%%%%%%%%%%%%%%%%%%%%%%%%%%%%%%%%%%%%%
\section{Localized measures}\label{s:localized}%%%%%%%%%%%%%%%%%%%%%%%
%%%%%%%%%%%%%%%%%%%%%%%%%%%%%%%%%%%%%%%%%%%%%%%%%%%%%%%%%%%%%%%%%%%%%%

\noindent
Equip the $\CAT(\kappa)$ space~$\MM$ with a probability measure~$\mu$.
This section sets forth conditions on~$\mu$ to ensure uniqueness of
its Fr\'echet mean and convexity of its Fr\'echet function~$F$.  More
details on convexity of Fr\'echet functions on $\CAT(\kappa)$ spaces
for $\kappa > 0$ can be found in~\cite{kuwae2014}; for the $\CAT(0)$ case,
see the standard reference~\cite{sturm2003}.

\begin{defn}[Localized]\label{d:localized}
A measure~$\mu$ on a $\CAT(\kappa)$ space~$\MM$ is
\begin{enumerate}
\item\label{i:punctual}%
\emph{punctual} if its Fr\'echet mean~$\bmu$ is unique and the
Fr\'echet function of~$\mu$ is locally convex in a neighborhood
of~$\bmu$;

\item\label{i:retractable}%
\emph{retractable} if the logarithm map $\log_\bmu: \MM \to \Tmu$
in Definition~\ref{d:log-map} is uniquely defined $\mu$-almost surely;

\item\label{i:localized}%
\emph{localized} if it is punctual and retractable.
\end{enumerate}
\end{defn}

\begin{example}\label{e:localized}
Geometric intuition behind Definition~\ref{d:localized} is that
% the measure
$\mu$ should be ``Fr\'echet-localized'', in the sense that
``retracting'' to the tangent cone
% at the Fr\'echet mean
at~$\bmu$ captures all of the mass.  For instance, if $\mu$ is
a~measure on a~$\CAT(\kappa)$ space~$\MM$ that is supported in a
metric ball $B(\bmu,R_\mu)$ of radius $R_\mu < R_\kappa =
\pi/\sqrt\kappa$---this can be any measure when $\kappa = 0$---then
$\mu$ is localized.  Indeed, thanks to results by Kuwae
\cite{kuwae2014}, the Fr\'echet mean~$\bmu$ of such a measure is
unique and the Fr\'echet function of~$\mu$ is $k$-uniform convex in a
small ball around~$\bmu$.  In addition, such a measure is retractable
because the cut locus (the closure of the set of points with more than
shortest path to~$\bmu$) has measure~$0$.
\end{example}

\begin{lemma}\label{l:hmu}
Any localized measure~$\mu$ can be pushed forward to a measure
$$%$$ $
  \hmu = (\log_\bmu)_\sharp\mu
$$%$$ $
on~$\Tmu$.  (This notation~$\hmu$ for the tangential pushforward
measure is used throughout.)
\end{lemma}

In a $\CAT(\kappa)$ space~$\MM$, the first variation formula
\cite[Theorem~4.5.6]{BBI01} implies that the squared distance function
is of class~$C^1$.
% (see also \cite[Proposition~1.20]{shadow-geom}
It is convenient to have a notation for this square-distance function
with one endpoint fixed.

\begin{defn}\label{d:half-square-distance}
Any point $w$ in a metric space~$\MM$ has \emph{half square-distance
function}
$$%$$ $
  \rho_w = \frac 12 \dd^2(w,\mathord{\,\cdot\,}).
$$%$$ $
\end{defn}

\begin{prop}\label{p:gradient-squared-distance-function}
For any point $w \in \MM$ in a $\CAT(\kappa)$ space, $\rho_w$ is
differentiable and
$$%$$ $
  \nablaq \rho_w = -\log_q w
$$%$$ $
in the sense that $\nablaq \rho_w(V) = -\<\log_q w,V\>$ for all $V \in
T_q\MM$.
\end{prop}
\begin{proof}
Let $V$ be a unit length tangent vector in $T_q\MM$ and $\gamma(t) =
\exp_w tV$.  It follows from the first variation formulat (see
\cite[Proposition~1.20]{shadow-geom})
% \ref{p:gradient-distance}
that
\begin{equation*}%\end{equation*}
\lim_{t\to 0} \frac{\dd^2(\gamma(t),q)-\dd^2(w,q)}{t}
  =
  -2\cos \angle(\log_qw,V) \dd(w,q).\qedhere
\end{equation*}
% Thus $\nablaq \rho_w = -\log_q w$ as desired.
\end{proof}

The exponential in the following definition
% does not need the smoothly stratified setting because it involves
% only a single tangent vector: it is merely
is a constant-speed geodesic whose initial tangent at~$p$ is~$V$.
% it is not necessary to know any properties of~$\exp_p$ as a map from
% all of~$\TpM$ to~$\MM$.

\begin{defn}\label{d:directional-derivative}
Fix a punctual measure~$\mu$ on a $\CAT(\kappa)$ space~$\MM$.  The
\emph{directional derivative} $\nabla_{\!p} F$ of the Fr\'echet
function~$F$ at~$p$ is
\begin{align*}%\end{align*}%$$
  \nabla_{\!p} F : \TpM & \to \RR
\\
                      V &\mapsto \frac{d}{dt} F(\exp_p tV)|_{t=0}.
\end{align*}%$$
If $\mu$ is also retractable, the \emph{directional derivative}
at~$\bmu$ of~$F_\hmu$, where $\hmu = (\log_\bmu)_\sharp\mu$, is
\begin{align*}%\end{align*}%$$
  \nablabmu F_\hmu:\Tmu & \to \RR
\\
                           V & \mapsto \frac{d}{dt} F_\hmu(tV)|_{t=0}.
\end{align*}%$$
\end{defn}

\begin{cor}\label{c:nablamu(F)}
Fix a localized measure~$\mu$ on a $\CAT(\kappa)$ space~$\MM$.  For $V
\in \Tmu$, integrating inner products
(Definition~\ref{d:inner-product}) against $\hmu =
(\log_\bmu)_\sharp\mu$ (Lemma~\ref{l:hmu})~yields
$$%$$ $
  \nablabmu F(V) = -\int_{\Tmu}\<W,V\>_\bmu \hmu(dW).
$$%$$ $
\end{cor}
\begin{proof}
The result follows from
Proposition~\ref{p:gradient-squared-distance-function}, given the
definitions of~$F$ in~\eqref{eq:FrechetF} and the half square-distance
function in Definition~\ref{d:half-square-distance}.
\end{proof}

%%%%%%%%%%%%%%%%%%%%%%%%%%%%%%%%%%%%%%%%%%%%%%%%%%%%%%%%%%%%%%%%%%%%%%
\subsection{Convexity of the Fr\'echet directional derivative}\label{b:conv_dd}
\mbox{}\medskip%%%%%%%%%%%%%%%%%%%%%%%%%%%%%%%%%%%%%%%%%%%%%%%%%%%%%%%

\noindent
Recall a simple and useful result in \cite{BBI01} that says the sum of
adjacent angles on a nonpositively curved space is at least~$\pi$.

\begin{lemma}[see~{\cite[Lemma~4.3.7]{BBI01}}]\label{l:sum_adj_angles}
Suppose that $\gamma:[0,t]\to \MM$ is a geodesic from $p$ to~$r$ in a
$\CAT(\kappa)$ space~$\MM$.  Let $q = \gamma(t_0)$ with $t_0 \in
(0,t)$ be an inner point in the geodesic $\gamma\bigl([0,t]\bigr)$.
Then for any point $s \in \MM$,
$$%$$ $
  \angle(\log_qp,\log_qs)+\angle(\log_qs,\log_qr) \geq \pi.
$$%$$ $
\end{lemma}

\begin{prop}\label{p:convexity-of_Df}
When $\mu$ is a localized measure on a $\CAT(\kappa)$ space~$\MM$, the
directional derivative $\nablabmu F$ from
Definition~\ref{d:directional-derivative} is a convex function on
$(\Tmu ,\dd_\bmu)$.
\end{prop}
\begin{proof}
For convenience, identify $\Tmu$ with its tangent cone at~$\bmu$ via
% Definition~\ref{d:MM-is-conical}.
Definition~\ref{d:log-map}.  Then $V = \exp_\bmu V \in \Tmu$ (as in a
few earlier locations, this exponential and others in the rest of this
proof do not need the smoothly stratified~setting)~and
$$%$$ $
  \nablabmu F(V)
  =
  -\int_{\Tmu}\<W,V\>_\bmu \hmu(dW)
$$%$$ $
by Corollary~\ref{c:nablamu(F)}.  Thus, for all tangent vectors $V$
at~$\bmu$,
\begin{equation}\label{eq:convexity-of_Df:-1}%\end{equation}
  \nablabmu F(V) = \nablabmu F_\hmu(V).
\end{equation}

Next proceed to show that $\nablabmu F_\hmu$ is convex on the geodesic
space $(\Tmu,\dd_\bmu)$.  For $V,W \in \Tmu$ satisfying
$\angle(V,W) < \pi$, denote the geodesic of constant speed in
$(\Tmu, \dd_\bmu)$ from~$V$ to~$W$~by
$$%$$ $
  U(s) \text{ for } s \in [0,1],
  \text{ so }
  U(0) = V
  \text{ and }
  U(1) = W.
$$%$$ $
It suffices to show that
\begin{equation}\label{eq:convexity-of_Df:0}%\end{equation}
  \nablabmu F_\hmu\bigl(U(s)\bigr)
  \leq
  (1-s)\nablabmu F(\exp_\bmu V) + s\nablabmu F_\hmu(W).
\end{equation}
It follows from Lemma~\ref{l:flat-triangle} that $tU(s)$ for $s \in
[0,1]$ is the geodesic from~$tV$ to~$tW$.  Since $(\Tmu, \dd_\bmu)$ is
nonpositively curved by Lemma~\ref{l:limit-log-preserve-npc}, the
function $F_\hmu$ is convex because it is the Fr\'echet function
of~$\hmu$.  Convexity of~$F_\hmu$ then implies
$$%$$ $
  (1-s)F_\hmu(tV) + sF_\hmu(tW)
  \geq
  F_\hmu\bigl(tU(s)\bigr) \text{ for all } t \in [0,1].
$$%$$ $
Combining this with the definition of directional derivative
(Definition~\ref{d:directional-derivative}) produces
\begin{align*}%\end{align*}%$$
\nablabmu F_\hmu\bigl(U(s)\bigr)
  & = \lim_{t \to 0} \frac{F_\hmu\bigl(tU(s)\bigr) - F_\hmu(\bmu)}{t}
\\
  &\leq\lim_{t \to 0} \frac{(1-s)\bigl(F_\hmu(tV) - F_\hmu(\bmu)\bigr)
    + s\bigl(F_\hmu(tW) - F_\hmu(\bmu)\bigr)}{t}
\\
  & = (1-s)\lim_{t \to 0} \frac{F_\hmu(tV) - F_\hmu(\bmu)}{t} +
      s\lim_{t \to 0} \frac{F_\hmu(tW) - F_\hmu(\bmu)}{t}
\\
  & = (1-s)\nablabmu F_\hmu(V) + s\nablabmu F_\hmu(W),
\end{align*}%$$
which proves~\eqref{eq:convexity-of_Df:0} and thus completes the
proof.
\end{proof}

The following consequence of the proof above allows most of the work
in this paper to be accomplished using~$F_\hmu$ instead of $\Fmu =
F_\mu$.

\begin{cor}\label{c:F-and-Ffwd-same-dd}
Fix a localized measure~$\mu$ on a $\CAT(\kappa)$ space~$\MM$.  For
any $V\in \Tmu$,
$$%$$ $
  \nablabmu F(V) = \nablabmu F_\hmu(V).
$$%$$ $
\end{cor}
\begin{proof}
This is \eqref{eq:convexity-of_Df:-1}.
\end{proof}

%%%%%%%%%%%%%%%%%%%%%%%%%%%%%%%%%%%%%%%%%%%%%%%%%%%%%%%%%%%%%%%%%%%%%%
\subsection{Escape and fluctuating cones}\label{b:fluctuating-cone}%%%
\mbox{}\medskip%%%%%%%%%%%%%%%%%%%%%%%%%%%%%%%%%%%%%%%%%%%%%%%%%%%%%%%

\noindent
Recall notation from Definition~\ref{d:directional-derivative}: $\hmu$
is the pushforward of $\mu$ under the logarithm map $\log_\bmu$ and
$F_\hmu$ is the Fr\'echet function of~$\hmu$.
Corollary~\ref{c:F-and-Ffwd-same-dd} observed that $\nablabmu F =
\nablabmu F_\hmu$.  Thus it is easier to understand $\nablabmu F$ by
studying properties of~$\nablabmu F_\hmu$~on~$\Tmu$.

\begin{lemma}\label{l:mean-of-hmu}
Fix a localized measure~$\mu$ on a $\CAT(\kappa)$ space~$\MM$.
Identify $\bmu$ with the apex of $\Tmu$.  Then $\bmu$ is the Fr\'echet
mean of~$\hmu$ on~$\Tmu$.
\end{lemma}
\begin{proof}
Since $\Tmu$ is nonpositively curved by
Lemma~\ref{l:limit-log-preserve-npc}, the Fr\'echet function~$F_\hmu$
is convex.  Thus it suffices to show that $\nablabmu F_\hmu$ is
nonnegative on~$\Tmu$.  Notice that $\nablabmu F$ is nonnegative
on~$\Tmu$ because $F$ is a convex function with minimizer~$\bmu$.  The
result then follows from Corollary~\ref{c:F-and-Ffwd-same-dd}.
\end{proof}

Variation of the population Fr\'echet mean~$\bmu$ induced by
perturbing the measure~$\mu$ can only occur along a restricted set of
directions.  See \cite{escape-vectors} for a full discussion of this
perspective and its relation to central limit theorems.  For now, here
is the definition of that restricted set of directions.

\begin{defn}[Escape cone]\label{d:escape-cone}
Fix a localized measure~$\mu$ on a
% smoothly stratified metric
$\CAT(\kappa)$ space~$\MM$.  The \emph{escape cone} of~$\mu$ is the
set~$\Emu$ of directions along which the directional derivative
(Definition~\ref{d:directional-derivative}) at~$\bmu$ of the Fr\'echet
function vanishes:
\begin{align*}%\end{align*}%$$
  \Emu &= \{X \in \Tmu \mid \nablabmu F(X) = 0\}.
\end{align*}%$$
\end{defn}

\begin{remark}\label{r:orthant-space-escape}
Barden and Le \cite[Definition~13]{barden-le2018} define escape cones
for orthant spaces,
% The special case of Definition~\ref{d:escape-cone} for orthant spaces,
which are $\CAT(0)$ gluings of Euclidean right-angled orthants
\cite{centroids}.
\end{remark}

\begin{lemma}\label{l:Emu=Ehmu}
Identifying $\XX = \Tmu$ in Definition~\ref{d:escape-cone} with
$T_\bmu\XX$ via Definition~\ref{d:log-map},
$$%$$ $
  \Emu = \Ehmu = \{X \in \Tmu \mid \nablabmu F_\hmu(X) = 0\}.
$$%$$ $
\end{lemma}
\begin{proof}
Apply Corollary~\ref{c:F-and-Ffwd-same-dd} to the display in
Definition~\ref{d:escape-cone}.
\end{proof}

The next result has intrinsic interest, but it is also half of the
reason why the fluctuating cone (Definition~\ref{d:fluctuating-cone})
is convex.

\begin{prop}\label{p:escape-is-convex}
If $\mu$ is a measure on a $\CAT(\kappa)$ space~$\MM$, then the escape
cone~$\Emu$ is a closed, path-connected, geodesically convex subcone
of $(\Tmu, \dd_\bmu)$ and $\Emu \cap \Smu$ is contained in one
connected component of~$\Smu$.
\end{prop}
\begin{proof}
Suppose first that in~$\Emu$ there are unit vectors $V,W \in \Emu$
such that $V$ and~$W$ lie in two different components of the unit
tangent sphere~$\Smu$.  Then any unit vector $Z \in \Tmu$ other than
$V,W$ lies in a different component from one of $V$ and~$W$.  Thus
$$%$$ $
  \angle(V,Z) + \angle(Z,W) > \pi.
$$%$$ $
Hence $\cos\angle(V,Z) < -\cos\angle(Z,W)$.  It follows, by
integrating~$Z$, that
$$%$$ $
  \int_{\Tmu}\norm{Z} \cos\angle(V,Z) \hmu(dZ)
  \,<\,
  -\!\int_{\Tmu}\norm{Z} \cos\angle(Z,W) \hmu(dZ).
$$%$$ $
Therefore
$$%$$ $
  \nabla_\bmu F_\hmu(V) > -\nabla_\bmu F_\hmu(W)
$$%$$ $
by Corollary~\ref{c:nablamu(F)}.  Replacing $F_\hmu$ by~$F$ via
Corollary~\ref{c:F-and-Ffwd-same-dd},
$$%$$ $
  \nabla_\bmu F(V) > -\nabla_\bmu F(W),
$$%$$ $
which is a contradiction because both sides vanish by
Definition~\ref{d:escape-cone}.

It remains to show $\Emu$ is closed and geodesically convex.  For
closed, use Lemma~\ref{l:inner-product-is-continuous} and
Corollary~\ref{c:nablamu(F)}.  For convex, let $V,W$ be unit-length
vectors in $\Emu$ and $\gamma(t)$ a unit-speed shortest path from $V$
to~$W$ in $\Tmu$.  What is needed is that $\gamma(t) \in \Emu$ for any
$t \in [0,\dd_\bmu (V,W)]$, or equivalently that $\nabla_\bmu
F\bigl(\gamma(t)\bigr) = 0$.  It suffices to assume that $\angle(V,W)
< \pi$ because the geodesic from~$V$ to~$W$ in~$\Tmu$ is the union of
the segments from the apex to the two vectors if $\angle(V,W) \geq
\pi$.  Since $\nabla_\bmu F$ is convex on $\Tmu$ by
Proposition~\ref{p:convexity-of_Df} and has minimizer~$\bmu$ so
$\nabla_\bmu F$ is nonnegative on~$\Tmu$,
\begin{equation*}%\end{equation*}%$$
  0
  \leq
  \nabla_\bmu F\bigl(\gamma(t)\bigr)
  \leq
  (1-t)\nabla_\bmu F(V) + t\nabla_\bmu F(W)
  =
  0.
  \qedhere
\end{equation*}%$$
\end{proof}

One final lemma (the rest of the section is definitions and remarks)
formalizes the observation that if the escape cone contains
diametrically opposed directions $V$ and~$W$, then all of the mass
in~$\MM$ behaves, angularly speaking, as if it lies in a single vector
space containing~$V$ and~$W$.  It is applied in the proof of
Proposition~\ref{p:fluctuating-cone_small} on the way to showing that
the fluctuating cone indeed embeds isometrically into a vector space
(Theorem~\ref{t:collapse}).  For the statement, recall the angular
metric~$\dd_s$ from Lemma~\ref{l:angular-metric} on the unit
tangent sphere~$S_\bmu$.

\begin{lemma}\label{l:inj-of_limit-log-on_Cmu_1}
Fix a localized measure~$\mu$ on a $\CAT(\kappa)$ space~$\MM$.  If
$V,W \in \Emu$ with $\angle(V,W) = \pi$, then $\dd_s(V,W) = \pi$ and
the set
$$%$$ $
  S
  =
  \{x\in\MM \mid \angle(\log_\bmu x,V) + \angle(\log_\bmu x,W) = \pi\}
$$%$$ $
has measure $\mu(S) = 1$; that is, under the pushforward measure $\hmu
= (\log_\bmu)_\sharp\mu$,
$$%$$ $
  \angle(X,V) + \angle(X,W) \overset{\text{a.s.}}= \pi.
$$%$$ $
\end{lemma}
\begin{proof}
Since $V,W\in \Emu$,
Corollary~\ref{c:nablamu(F)} and Definition~\ref{d:escape-cone}
imply that
$$%$$ $
  \int_\MM \<\log_\bmu x,V\>_\bmu \mu(dx)
  =
  \int_\MM \<\log_\bmu x,W\>_\bmu \mu(dx)
  =
  0.
$$%$$ $
On the other hand, $\angle(V,W) = \pi$, so $\dd_s(V,W) = \theta \geq
\pi$.  Thus, for any $x \in \MM$,
$$%$$ $
  \dd_s(\log_\bmu x,V)\geq \theta -\dd_s(\log_\bmu x,W) \geq \pi
  -\dd_s(\log_\bmu x,W).
$$%$$ $
Therefore
$$%$$ $
  \int_\MM \<\log_\bmu x, V\>_\bmu \mu(dx)
  \leq
  \int_\MM \<\log_\bmu x, W\>_\bmu \mu(dx)
  =
  0.
$$%$$ $
Combining these three displayed equations (or inequalities) yields
$\dd_\bmu (V,W) = \pi$ and $ \angle(\log_\bmu x,V)+\angle(\log_\bmu
x,W) = \pi $ for $\mu$-almost all~$x$.
\end{proof}

\begin{defn}[Hull]\label{d:hull}
Given a subset $\cS \subseteq \XX$ of a
% smoothly stratified metric
$\CAT(0)$ conical space~$\XX$, the \emph{hull} of~$\cS$ is the
smallest geodesically convex cone $\hull\cS \subseteq \XX$
containing~$\cS$.  For a localized measure~$\mu$ on~$\MM$, set
$$%$$ $
  \hull\mu
% = \hull\hmu
  = \hull\supp(\hmu),
$$%$$ $
the hull in~$\Tmu$ of the support of the pushforward measure $\hmu =
(\log_\bmu)_\sharp\mu$ in~Lemma~\ref{l:hmu}.
\end{defn}

\begin{defn}[Fluctuating cone]\label{d:fluctuating-cone}
Fix a localized measure~$\mu$ on a $\CAT(\kappa)$ space~$\MM$.  The
\emph{fluctuating cone} of~$\mu$ is the intersection
\begin{align*}%\end{align*}%$$
  \Cmu &= \Emu \cap \hull\mu
  \\   &= \{V \in \hull\mu \mid \nablabmu F(V) = 0\}
\end{align*}%$$
of the escape cone and hull of~$\mu$ from
Definitions~\ref{d:escape-cone} and~\ref{d:hull}.  Let $\oCmu$ be its
closure.
\end{defn}

\begin{lemma}\label{l:Cmu=Ehmu-intersect-hull}
In Definition~\ref{d:fluctuating-cone}, $\Ehmu$ can be used instead
of~$\Emu$:
\begin{align*}%\end{align*}%$$
  \Chmu = \Cmu &= \Ehmu \cap \hull\mu
\\             &= \{V \in \hull\mu \mid \nablabmu F_\hmu(V) = 0\}
\end{align*}%$$
once $\XX = \Tmu$ is identified with $T_\bmu\XX$ via
Definition~\ref{d:log-map}.
\end{lemma}
\begin{proof}
Apply Lemma~\ref{l:Emu=Ehmu} to the display in
Definition~\ref{d:fluctuating-cone} and note that $\hull\mu =
\hull\hmu$ by Definition~\ref{d:hull} once $\XX = \Tmu$ is identified
with $T_\bmu\XX$.
\end{proof}

\begin{remark}\label{r:tripod}
The purpose of the fluctuating cone~$\Cmu$ is to encapsulate those
directions in which the Fr\'echet mean $\bmu$ can be induced to wiggle
by adding to~$\mu$ a point mass in~$\MM$ along that direction (this is
made precise in the main theorem of \cite[Section~4]{escape-vectors});
% specifically \ref{t:escape-vector-confinement}
hence the terminology.  However, if the measure~$\mu$ is supported on
a ``thin'' subset of~$\MM$,
% such as the spine in an open book (cf.~\cite{hotz-et-al.2013}),
then it is possible to induce fluctuations in directions that have
nothing whatsoever to do with the geometry in~$\MM$ of the (support
of)~$\mu$ by adding a point mass outside~$\supp\mu$; see the next
Example.
% Do to Ezra: We want to continue your line here ``However, if the
% measure $\mu$ is supported on a ``thin'' subset of M, then it is
% possible to induce fluctuations in directions that have nothing
% whatsoever to do with the geometry in M of the (support of the)
% measure $\mu$'' with --- ``but doing so requires the Fr\'echet
% function of $\mu$ to have singular Hessian at the mean (Theorem 8
% and Theorem 10 in \cite{TEH2021GSI}), which violates the strong
% convexity property of $F$ in Lemma~2.47.''
% \cite{TEH2021GSI} = \cite{tran-eltzner-huckemann2021}, I think
% The remedy for this is to define
That is why the fluctuating cone is assumed to lie within the convex
hull of the support of~$\hmu$: only fluctuations of $\bmu$ that can be
realized---at least in principle---by means of samples from~$\mu$
itself
% (or from the convex hull of the support of~$\mu$)
% note that any point in the convex hull of~$\mu$ is relevant
% for this purpose, since adding a mixture of discrete measures at
% points of~$\mu$ is effectively the same as adding a discrete measure
% at their mean, which can be an arbitrary point in the convex hull
are relevant to the asymptotics in~a~central~limit~theorem.
\end{remark}

\begin{example}\label{e:tripod}
For a concrete example, consider a measure~$\mu$ supported on the
spine~$R$ of an open book~$\MM$ (see~\cite{hotz-et-al.2013}).  The
spine~$R$ is simply a vector space, where the usual central limit
theorem yields convergence to a Gaussian supported on~$R$.  It is true
that the Fr\'echet mean $\bmu$ can be induced to fluctuate off of~$R$
onto any desired page of~$\MM$ by adding a point mass on the relevant
page, but that observation is irrelevant to the CLT, which only cares
about fluctuations of~$\bmu$ within~$R$.
\end{example}

%%%%%%%%%%%%%%%%%%%%%%%%%%%%%%%%%%%%%%%%%%%%%%%%%%%%%%%%%%%%%%%%%%%%%%
\subsection{Measures and Fr\'echet means under limit log maps}\label{b:mu-lim-log}
\mbox{}\medskip%%%%%%%%%%%%%%%%%%%%%%%%%%%%%%%%%%%%%%%%%%%%%%%%%%%%%%%

\noindent
This section collects three essential results about behavior of
measures and their Fr\'echet means under limit log maps.

\begin{hyp}\label{h:XX}
In this subsection, $\MM$ is a $\CAT(\kappa)$ space endowed with a
localized measure~$\mu$.  Fix a unit vector $Z$ in the escape
cone~$\Emu$ (Definition~\ref{d:escape-cone}).  Denote the pushforward
of~$\hmu = (\log_\bmu)_\sharp\mu$ under limit log along~$Z$~by
$$%$$ $
  \mu_Z = (\LL_Z)_\sharp\hmu.
$$%$$ $
A~subscript $z$ indicates transition to a vector in the limit tangent
cone~$\TZM$; for example if $\XX = \Tmu$ and $V \in T_\bmu\XX$ or
(see~Remark~\ref{r:TvM-vs-TvX}) $V \in \Tmu$, then $V_z = \LL_Z(V)$.
\end{hyp}

\begin{remark}\label{r:XX}
In applications of the results in this subsection, $\MM$ is usually a
$\CAT(0)$ cone~$\XX = \Tmux$ (see~Remark~\ref{r:TvM-vs-TvX}) where
$\mu = \hmu$ is automatically localized by Example~\ref{e:localized}
and has Fr\'echet mean the apex~$\bmu$.  At the outset, the reader
should think of~$\XX$ as~$\Tmu$ endowed with the measure $\hmu =
(\log_\bmu)_\sharp\mu$ from Lemma~\ref{l:hmu}, whose mean is~$\bmu$ by
Lemma~\ref{l:mean-of-hmu}.  In that case, $\Ehmu = \Emu$ by
Lemma~\ref{l:Emu=Ehmu}.  But in Section~\ref{s:collapse}, limit logs
are iterated, so it is important to consider $\XX = \TZM$ or $\XX =
\vT_{Z'}\TZM$, and so on.
\end{remark}

\subsubsection{Fr\'echet mean preservation}\label{ss:mean-preserv}%%%%
\mbox{}\medskip

\noindent
The first target of this section,
Proposition~\ref{p:limit-log-preserve-mean}, says that Fr\'echet mean
is preserved under limit log along any direction in the escape
cone~$\Emu$.  This is analogous to the property of the log map in
manifold cases, namely that the Fr\'echet mean is preserved under the
log map there.  Preservation of Fr\'echet means is crucial for the
collapse by iterative d\'evissage in Section~\ref{s:collapse}.

Recall the unit tangent sphere $\Smu$
(Definition~\ref{d:tangent-cone}) at the Fr\'echet mean~$\bmu$.

\begin{lemma}\label{l:bound-near-shadow}
Let
$$%$$ $
  \II(W) = \bigl\{V \in \Smu \mid \angle(V,W) > \angle(V_z,W_z)\bigr\}.
$$%$$ $
Then $\angle(V,W) = \angle(V_z,W_z)$ for $V \in \Smu \setminus
\II(W)$.  In contrast, if\/ $V \in \II(W)$, then
\begin{equation}\label{eq:lem:bound_near_shadow:2}%\end{equation}
  \pi
  \geq
  \angle(W,V)
  \geq
  \angle(W_z,V_z)
  \geq
  \pi-\angle(W,Z)
\end{equation}
and
\begin{equation}\label{eq:lem:bound_near_shadow:3}%\end{equation}
  \angle(W,V) - \angle(W_z,V_z) \leq \angle(W,Z).
\end{equation}
\end{lemma}
\begin{proof}
All angles are bounded above by~$\pi$ by Definition~\ref{d:angle}.
Since $\angle(W,V) \geq \angle(W_z,V_z)$ by
Proposition~\ref{p:limit-log-contracts},
$$%$$ $
  \angle(V,W) = \angle(V_z,W_z)
$$%$$ $
when $V \in \Smu \setminus \II(W)$.  It remains to
show~\eqref{eq:lem:bound_near_shadow:2}
and~\eqref{eq:lem:bound_near_shadow:3}.  Suppose that $V\in \II(W)$.
It follows from Proposition~\ref{p:limit-log-iso-on-small-geo} that
the geodesic $\gamma \subseteq \Smu$ from~$W$ to~$V$ contains at least
two distinct points whose angle witn~$Z$ is~$\pi$.  Let $Y$ be the
first point along~$\gamma$ whose angle with~$Z$ is~$\pi$.
Proposition~\ref{p:limit-log-iso-on-small-geo} again yields
$\angle(W,Y) = \angle(W_z,Y_z)$.  Thus
$$%$$ $
  \pi
  \geq
  \angle(W,V)
  \geq
  \angle(W_z,V_z)
  \geq
  \angle(W_z,Y_z)
  =
  \angle(W,Y).
\pagebreak[2]
$$%$$ $
On the other hand, the triangle inequality in $\Smu$ implies
\begin{equation}\label{eq:lem:bound_near_shadow:5}%\end{equation}
  \angle(W,Y)
  \geq
  \angle(Y,Z) - \angle(W,Z)
  =
  \angle(Y, Z) - \angle(W, Z)
  =
  \pi - \angle(W,Z).
\end{equation}
Combining \eqref{eq:lem:bound_near_shadow:5} and the display preceding
it yields~\eqref{eq:lem:bound_near_shadow:2}.

To see~\eqref{eq:lem:bound_near_shadow:3}, notice
from~\eqref{eq:lem:bound_near_shadow:5} that $\angle(W,Y) \geq \pi -
\angle(W,Z)$, so
\begin{align*}%\end{align*}%$$
  \angle(W,V)-\angle(W_z,V_z)
   &\leq \angle(W,V) - \angle(W_z,Y_z)
\\* & =   \angle(W,V) - \angle(W,Y)
\\* &\leq \pi - (\pi - \angle(W,Z)) = \angle(W,Z).\qedhere
\end{align*}%$$
\end{proof}

\begin{lemma}\label{l:angular-dev-of-nablaF}
Fix a unit vector $U_z \in \TZX$ such that $\angle(U_z,Z_z) =
\alpha < \pi$.  Let $\wh\gamma(t)$ for $t \in [0,1]$ be the
constant-speed geodesic from $Z_z$ to~$U_z$ in~$\LL_Z(\Smu)$.
Then
$$%$$ $
  \lim_{t\to 0}\frac{1}{\sin(t\alpha)}
    \Bigl(\nablabmu F_\mu\bigl(\LL^{-1}_Z(\wh\gamma(t))\bigr)
    -
    \nablabmu F_{\mu_Z}\bigl(\wh\gamma(t)\bigr)\!\Bigr)
  =
  0.
$$%$$ $
\end{lemma}
\begin{proof}
Set $W_z = \wh\gamma(t)$ and $W = \LL^{-1}(W_z)$.  Thanks to
Proposition~\ref{p:limit-log-iso-on-small-geo}, $W$ is well defined
since $\angle(W_z, Z_z) < \angle(U_z, Z_z) < \pi$.

The proof requires estimating the difference $\nablabmu F_\mu(W) -
\nablabmu F_{\mu_Z}(W_z)$.  Recall from
Proposition~\ref{p:limit-log-contracts} that $\angle(W,V) \geq
\angle(W_z,V_z)$ for all $V \in \Tmu$, so
\begin{equation}\label{eq:nabla-difference}%\end{equation}
  \nablabmu F_\mu(W) - \nablabmu F_{\mu_Z}(W_z) \geq 0.
\end{equation}
Thus
\begin{align*}%\end{align*}%$$
  \lim_{t\to 0}\frac{1}{\sin(t\alpha)}
    \Bigl(\nablabmu F_\mu\bigl(\LL^{-1}_Z(\wh\gamma(t))\bigr)
    \!-\!
    \nablabmu F_{\mu_Z}\bigl(\wh\gamma(t)\bigr)\!\Bigr)
&= \lim_{t\to 0}\frac{1}{\sin(t\alpha)}
   \bigl(\nablabmu F_\mu(W) \!-\! \nablabmu F_{\mu_Z}(W_z)\bigr)
\\[.5ex]
&\geq 0.
\end{align*}%$$

It remains to show that
$$%$$ $
  \lim_{t\to 0}\frac{1}{\sin(t\alpha)}
  \bigl(\nablabmu F_\mu(W) \!-\! \nablabmu F_{\mu_Z}(W_z)\bigr)
  \leq
  0.
$$%$$ $
As in Lemma~\ref{l:bound-near-shadow} but not restricted to unit
vectors, let
$$%$$ $
  \JJ(W)
  =
  \bigl\{V \in \Tmu \mid \angle(V,W) > \angle(V_z,W_z)\bigr\}.
$$%$$ $
Lemma~\ref{l:bound-near-shadow} concludes that $\angle(V,W) =
\angle(V_z,W_z)$ when $V \in \Tmu \setminus \JJ(W)$, while
\begin{align}\label{eq:lem:angular_dev-of_nablaF:3}%\end{align}
& \angle(W,V) - \angle(W_z,V_z)
  \leq
  \angle(W_z, Z_z)
  =
  t\alpha
\\\nonumber\text{and }
& \pi
  \geq
  \angle(W,V)
  \geq
  \angle(W_z,V_z)
  \geq
  \pi-t\alpha
\end{align}
when $V \!\in \JJ(W)$.  Expressing each term in the left side
of~\eqref{eq:nabla-difference} with
Corollary~\ref{c:nablamu(F)}~yields
\begin{align*}%\end{align*}%$$
\nablabmu F_{\mu_Z}(W_z)\,
  &= \,-\!\int_{\TZM}\<W_z,V_z\>\mu_Z(dV_z)
\\*
  &= \,-\!\int_{\Tmu}\<W_z,V_z\>\hmu(dV)
\\*
  &= \,-\!\int_{\Tmu\setminus \JJ(W)} \<W_z,V_z\> \hmu(dV)
       -  \int_{\JJ(W)}\<W_z,V_z\> \hmu(dV)
\\
  &= \,-\!\int_{\Tmu\setminus \JJ(W)} \<W,V\> \hmu(dV)
       -  \int_{\JJ(W)}\<W_z,V_z\> \hmu(dV)
\\[1.5ex]\makebox[0pt][r]{\text{and}}\quad\ \
\nablabmu F_\mu (W)\,
  &= \,-\!\int_{\Tmu} \<W,V\> \hmu(dV)
\\*
  &= \,-\!\int_{\Tmu \setminus \JJ(W)} \<W,V\> \hmu(dV)
       -  \int_{\JJ(W)}\<W,V\> \hmu(dV).
\end{align*}%$$
The first integrals in these expressions for $\nablabmu
F_{\mu_Z}(W_z)$ and $\nablabmu F_\mu (W)$, namely the integrals over
$\Tmu \setminus \JJ(W)$, are equal.  Subtracting both sides
therefore produces
\begin{align*}%\end{align*}%$$
\nablabmu
  &F_\mu (W) - \nablabmu F_{\mu_Z}(W_z)
\\*
  &= \int_{\JJ(W)}\bigl(\<W_z,V_z\> - \<W,V\>\bigr) \hmu(dV)
\\
  &= \int_{\JJ(W)}\|V\|\Bigl(\cos\bigl(\angle(W_z,V_z)\bigr)
                       - \cos\bigl(\angle(W,V)\bigr)\!\Bigr) \hmu(dV)
\\
  &= \int_{\JJ(W)} - 2\|V\|\sin\Bigl(\frac{\angle(W_z,V_z) - \angle(W,V)}2\Bigr)
     \sin\Bigl(\frac{\angle(W_z,V_z) + \angle(W,V)}2\Bigr) \hmu(dV)
\\
  &\leq \int_{\JJ(W)}2\|V\|\sin(t\alpha/2)\sin(\pi -t\alpha) \hmu(dV)
   \qquad\text{by }\eqref{eq:lem:angular_dev-of_nablaF:3}
\\
  &\leq C \sin(t\alpha/2) \sin(t\alpha)
   \qquad\text{ for } C \geq \int_{\Tmu} \|V\| \hmu(dV) < \infty.
\end{align*}%$$
Therefore\pagebreak[2]
\begin{align*}%\end{align*}%$$
  \lim_{t\to 0}\frac{1}{\sin(t\alpha)}
    \Bigl(\nablabmu F_\mu\bigl(\LL^{-1}_Z(\wh\gamma(t))\bigr)
    \!-\!
    \nablabmu F_{\mu_Z}\bigl(\wh\gamma(t)\bigr)\!\Bigr)
  &= \lim_{t\to 0} \frac{1}{\sin (t\alpha)}
     \bigl(\nablabmu F_\mu (W) - \nablabmu F_{\mu_Z}(W_z)\bigr)
\\*
  &\leq C \lim_{t\to 0} \frac{1}{\sin(t\alpha)} \sin(t\alpha/2) \sin(t\alpha)
\\*
  &\leq C \lim_{t\to 0} \sin(t\alpha/2) = 0.\qedhere
\end{align*}%$$
\end{proof}

\begin{prop}\label{p:limit-log-preserve-mean}
Under Hypotheses~\ref{h:XX}, $\bmu_Z = \LL_Z (\bmu)$ is the Fr\'echet
mean of~$\mu_Z$ on~$\TZM$, where $\bmu$ is identified with the apexes
of both~$\Tmu$ and~$\TZM$.
\end{prop}
\begin{proof}
By Proposition~\ref{p:limit-log-contracts}
% Corollary~\ref{c:LL-preserves-angle-to-Z}
$\angle(Z,\hspace{-.5pt}V) = \angle(Z_z,\hspace{-.5pt}V_z)$ for all $V
\!\in \Tmu$, so Corollary~\ref{c:nablamu(F)}~yields
\begin{equation}\label{eq:limit-log_preserve_mean:1}%\end{equation}
\begin{split}
  \nablabmu F_{\mu_Z}(Z_z)
  &= -\!\int_{\TZM} \<Z,V_z\> \mu_Z(dV_z)
\\
  &= -\!\int_{\Tmu} \<Z,V_z\> \hmu(dV)
\\
  &= -\!\int_{\Tmu} \<Z,V\> \hmu(dV)
\\
  &= \nablabmu F_\hmu(Z)
\\
  &= 0.
\end{split}
\end{equation}
Because $\TZM$ is a $\CAT(0)$ space by
Lemma~\ref{l:limit-log-preserve-npc}, it suffices to show that
\begin{equation}\label{eq:limit-log_preserve_mean:2}%\end{equation}
  \nablabmu F_{\mu_Z}(V_z)\geq 0 \text{ for all } V_z \in \TZM.
\end{equation}
Assume, for contradiction, that
it does not hold, so there is some $V_z \in \TZM$ with
\begin{equation}\label{eq:limit-log_preserve_mean:2_1}%\end{equation}
  \nablabmu F_{\mu_Z}(V_z) = -\delta < 0.
\end{equation}

The first consequence of this assumption is that $\angle(V_z,Z_z) <
\pi$.  Indeed,
\begin{align*}%\end{align*}%$$
\dis
\angle(V_z,Z_z) \geq \pi
   &\dis\ \implies
\begin{array}[t]{@{\ }l@{\ }l@{}}
\dis
\angle(V_z,W_z) + \angle(W_z,Z_z)
  &\dis\geq \angle(V_z,Z_z)
\\[1ex]
  &\dis\geq \pi \text{ for all } W_z \in \TZM
\end{array}
\\ &\dis\ \implies
\begin{array}[t]{@{\ }l@{\ }l@{}}
\dis
-\nablabmu F_{\mu_Z}(Z_z)
  &\dis  =  \int_{\TZM} \<Z_z,W_z\> \mu_Z(dW_z)
\\[3ex]\dis
  &\dis\leq -\!\int_{\TZM} \<V_z,W_z\> \mu_Z(dW_z)
\\[3ex]\dis
  &\dis  =  \nablabmu F_{\mu_Z}(V_z)
\end{array}
\\[1ex] &\dis\ \implies
\begin{array}[t]{@{\ }l@{\ }l@{}}
\dis
0
  &\dis  =  \nablabmu F_{\mu_Z}(Z_z)
\\[1ex]\dis
  &\dis\geq -\nablabmu F_{\mu_Z}(V_z)
\\[1ex]\dis
  &\dis  =  \delta > 0.
\end{array}
\end{align*}%$$

As a result, $\angle(V,Z) = \angle(V_z,Z_z) < \pi$.  Thanks to the
angular metric in Lemma~\ref{l:angular-metric}, this implies there is
a geodesic in~$\LL_Z(\Smu)$ from $V_z$ to~$Z_z$.  Let $\gamma(t)$ and
$\wh\gamma(t)$ for $t \in [0,1]$ be geodesics from $Z_z$ to~$V_z$
in~$\TZM$ and in~$\LL_Z(\Smu)$, respectively.

\begin{figure*}[!ht]\label{f:gammas}
\centering
\includegraphics[width=0.35\textwidth]{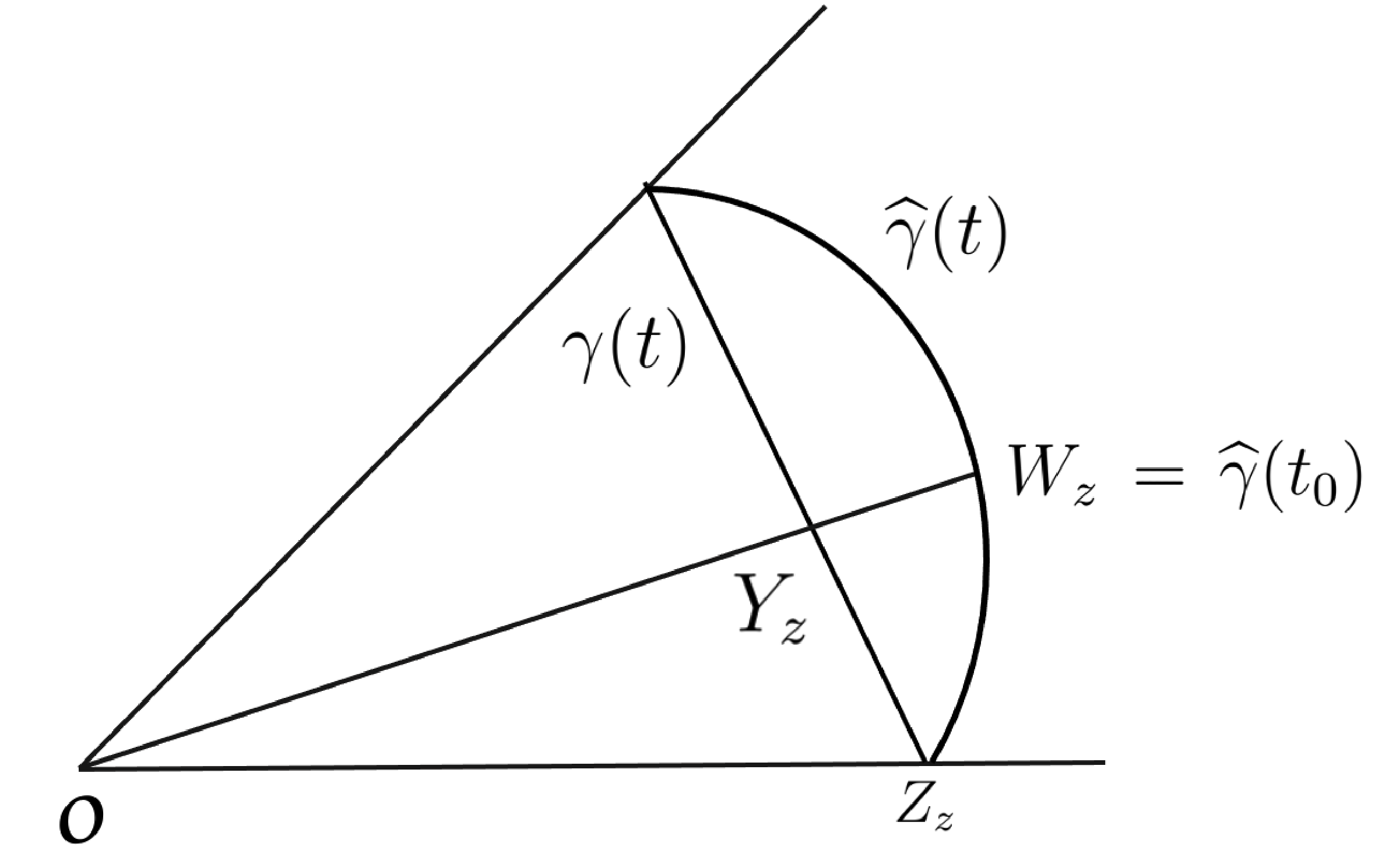}
\caption{}
\end{figure*}

Proposition~\ref{p:convexity-of_Df} applied to~$\mu_Z$ (this is where
the localized hypothesis enters) implies that $\nablabmu
F_{\mu_Z}$ is convex on~$\TZM$.  Thus, for all $t\in [0,1]$,
\begin{equation}\label{eq:limit-log_preserve_mean:3}%\end{equation}
  \nablabmu F_{\mu_Z}\bigl(\gamma(t)\bigr)
  \leq
  t\nablabmu F_{\mu_Z}(V_z) + (1-t)\nablabmu F_{\mu_Z}(Z_z)
  =
  t(-\delta)
  <
  0.
\end{equation}
Due to Lemma~\ref{l:flat-triangle} the triangle $\triangle \bmu V_z
Z_z$ is flat, so the line segment from~$\bmu$ to~$\gamma(t)$ meets
$\wh\gamma$ at some point.  It follows that $\nablabmu
F_{\mu_Z}\bigl(\wh\gamma(t)\bigr) < 0$ for all $t \in [0,1]$.

Let $t_0 \in [0,1]$ and $W_z = \wh\gamma(t_0)$.  Denote by~$Y_z$ the
intersection of~$\gamma$ with the line segment~$\bmu W_z$.  Then
\begin{equation}\label{eq:limit-log_preserve_mean:4}%\end{equation}
  \|Y_z\| \nablabmu F_{\mu_Z}(W_z) = \nablabmu F_{\mu_Z}(Y_z).
\end{equation}
Write $\angle(V_z,Z_z) = \alpha$ and $\angle(W_z,Z_z) = x$, so $t_0 =
\frac x\alpha$.  Elementary computations reveal
\begin{equation}\label{eq:limit-log_preserve_mean:5}%\end{equation}
\begin{split}
  \frac{|Y_z Z_z|}{|V_z Z_z|}
  &= \frac{\sin x}{\sin x + \sin(\alpha-x)}
   = \frac{\sin(t_0\alpha)}{\sin(t_0\alpha) + \sin(\alpha(1-t_0))},
  \text{ and }
\\
  \|Y_z\|
  &= \frac{\sin \alpha}{\sin x + \sin(\alpha-x)}
   = \frac{\sin(\alpha)}{\sin(t_0\alpha) + \sin(\alpha(1-t_0))},
\end{split}
\end{equation}
where $|Y_z Z_z|$ is the geodesic distance between $Y_z$ and~$Z_z$.
Setting $a = \frac{|Y_z Z_z|}{|V_z Z_z|}$ and applying convexity of
$\nablabmu F_{\mu_Z}$ on the line segment $V_z Z_z$ gives
$$%$$ $
  \nablabmu F_{\mu_Z}(Y_z)
  \leq
  \frac{|Y_z Z_z|}{|V_z Z_z|} \nablabmu F_{\mu_Z}(V_z)
  + \frac{|Y_zV_z|}{|V_z Z_z|}\nablabmu F_{\mu_Z}(Z_z)
  =
  -a\delta.
$$%$$ $
Combining this with \eqref{eq:limit-log_preserve_mean:4} produces
\begin{align*}%\end{align*}%$$
  \delta
  & \leq -\frac{\|Y_z\|}{a}\nablabmu F_{\mu_Z}(W_z)
\\
  &   =  -\frac{\|Y_z\|}{a}
         \bigl(\nablabmu F_{\mu_Z}(W_z) - \nablabmu F_\mu(W)\bigr)
         -\frac{\|Y_z\|}{a}
         \nablabmu F_\mu (W)
\\
  & \leq -\frac{\|Y_z\|}{a}
         \bigl(\nablabmu F_{\mu_Z}(W_z)-\nablabmu F_\mu(W)\bigr)
    \quad\text{because}\quad
    \nablabmu F_\mu(V) > 0 \text{ for all } V \in \Tmu.
\end{align*}%$$
Substituting into this the formulas for $\|Y_z\|$ and $a$ from
\eqref{eq:limit-log_preserve_mean:5} reveals
% ooh -- good word
\begin{equation}\label{eq:limit-log_preserve_mean:8}%\end{equation}
  \begin{split}
    \delta \leq -\frac{\sin \alpha}{\sin
    (t_0\alpha)}\left(\nablabmu F_{\mu_Z}(W_z)-\nablabmu F_\mu (W)
    \right).
  \end{split}
\end{equation}
It follows from Lemma~\ref{l:angular-dev-of-nablaF} that
\begin{align*}%\end{align*}%$$
  \lim_{t_0\to 0}\frac{1}{\sin(t_0\alpha)}
    &\bigl(\nablabmu F_\mu(W) - \nablabmu F_{\mu_Z}(W_z)\bigr)
\\* & = \lim_{t\to 0}\frac{1}{\sin(t\alpha)}
        \Bigl(\nablabmu F_\mu \bigl(\LL^{-1}_Z\bigl(\wh\gamma(t)\bigr)\bigr)
        - \nablabmu F_{\mu_Z}\bigl(\wh\gamma(t)\bigr)\Bigr)
\\* &=0.
\end{align*}%$$
Applying this after letting $t_0$ converge to~$0$ in
\eqref{eq:limit-log_preserve_mean:8} gives
$ %$$ $
  \delta \leq 0,
$ %$$ $
which contradicts the assumption
\eqref{eq:limit-log_preserve_mean:2_1}.  Thus
\eqref{eq:limit-log_preserve_mean:2} holds and the proof is
complete.
\end{proof}

\subsubsection{Escape cone preservation}\label{ss:escape-preserv}%%%%%
%mbox{}\medskip

\begin{prop}\label{p:LL-preserves-Emu}
Under Hypotheses~\ref{h:XX}, the escape cone~$\Emuz$ of~$\mu_Z$
contains the limit log along~$Z$ of the escape cone of~$\mu$; that is,
$\LL_Z(\Emu) \subseteq \Emuz$.
\end{prop}
\begin{proof}
Since $\bmu_Z$ is the Fr\'echet mean of~$\mu_Z$ by
Proposition~\ref{p:limit-log-preserve-mean},
$$%$$ $
  \Emuz
  =
  \bigl\{V\in T_{\bmu_Z}(\LL_Z \MM) \mid \nabla_{\!\bmu_Z}F_{\mu_Z}(V) = 0\bigr\},
$$%$$ $
where $F_Z$ is the Fr\'echet function of $\mu_Z$.  In addition, from
Definition~\ref{d:limit-log} (of limit log map) and the fact that
$\bmu_Z$ is the Fr\'echet mean of $\mu_Z$, any $V \in \Emu$ satisfies
$$%$$ $
  0
  \leq
  \nabla_{\!\bmu_Z} F_{\mu_Z}\bigl(\LL_Z(V)\bigr)
  \leq
  \nablabmu F(V)
  =
  0,
$$%$$ $
because limit log maps contract
(Proposition~\ref{p:limit-log-contracts}).  Hence $\LL_Z(\Emu)
\subseteq \Emuz$.
\end{proof}

\begin{cor}\label{c:LL-preserves-Cmu}
Under Hypotheses~\ref{h:XX}, the fluctuating cone~$\Cmuz$ of~$\mu_Z$
contains the limit log along~$Z$ of the fluctuating cone of~$\mu$;
that is, $\LL_Z(\Cmu) \subseteq \Cmuz$.
\end{cor}
\begin{proof}
Combine preservation of~$\Emu$ in Proposition~\ref{p:LL-preserves-Emu}
and preservation of convex cones in
Corollary~\ref{c:hull-preserved-by-limit-log} applied to $\cS =
\supp\hmu$, since $\Cmu = \Emu \cap \hull\mu$ by
Definition~\ref{d:fluctuating-cone}.
\end{proof}

\subsubsection{Fluctuating cone isometry}\label{ss:fluctuating-isometry}
\mbox{}\medskip

\noindent
The next result asserts an unexpected and decisive confinement: the
fluctuating cone~$\Cmu$ from Definition~\ref{d:fluctuating-cone} is
contained within a sector of total angle at most~$2\pi$.  Moreover,
this confinement is valid when viewed from the perspective of any
vector in the escape cone~$\Emu$ from Definition~\ref{d:escape-cone}.

\begin{prop}\label{p:fluctuating-cone_small}
Under Hypotheses~\ref{h:XX}, at most one unit vector $X \in \Cmu$
satisfies $\angle(Z,X) \geq \pi$.
\end{prop}

% The next result contributes to the proof that limit log maps induce
% isometric embeddings on fluctuating cones; see
% Proposition~\ref{p:fluctuating-cone_small} and
% Corollary~\ref{c:inj-of-limit-log-on-Cmu}.
The proof appeals to a lemma.  Recall the angular metric~$\dd_s$ from
Lemma~\ref{l:angular-metric} on the unit tangent sphere~$S_\bmu$.

\begin{lemma}\label{l:K-is-convex}
Fix a point~$\bmu$ in a smoothly stratified metric space~$\MM$.  Let
$Z \in \Smu$ and assume $X$ is a unit tangent vector at~$\bmu$
satisfying $\angle(Z,X) = \pi$.  Then
$$%$$ $
  \KK
  =
  \bigl\{Y \in \Tmu \mid
    \angle(Y,Z) + \angle(Y,X)
    =
    \pi
  \bigr\}
$$%$$ $
is a convex subcone
% ah: this bit about proper containment isn't true when there's just
% one~$X$, but if there are two---yielding two convex cones $\KK_0$
% and $\KK_1$, for example---then $\KK_0$ can't equal~$\KK_1$ because
% $X_0 \in \KK_0 \setminus \KK_1$; this isn't an issue for the
% application of this Lemma in the proof of
% Proposition~\ref{p:fluctuating-cone_small}, because there the convex
% cone $\KK$ there is indeed an intersection $\KK_0 \cap \KK_1$
of~$\Tmu$.  Equivalently, if
$$%$$ $
  \Delta
  =
  \bigl\{V \in \Smu \mid
    \angle(V,Z) + \angle(V,X)
    =
    \pi
  \bigr\},
$$%$$ $
so $\KK$ is the Euclidean cone over its set~$\Delta$ of unit vectors,
then for any $U,V\in \Delta$ with $\dd_s(U,V) < \pi$ the geodesic
in~$S_\bmu$ from $U$ to~$V$ lies in~$\Delta$.
\end{lemma}
\begin{proof}
Because $\angle(Z,X) = \pi$, any shortest path
$\zeta\bigl([0,1]\bigr)$ in~$S_\bmu$ from~$Z$ to a vector $V \in
\Delta$ is the initial segment of a shortest path
$\sigma\bigl([0,1]\bigr)$ from $Z$ to~$X$; that is,
$$%$$ $
  \zeta\bigl([0,1]\bigr)
  \subseteq
  \sigma\bigl([0,1]\bigr)
  \text{ and } \zeta(0) = \sigma(0) = Z.
$$%$$ $
Let $\xi\bigl([0,1]\bigr)$ be a shortest path in~$S_\bmu$ from $Z$ to
another vector $U \in \Delta$ with $\dd_s(U,V) <\nolinebreak \pi$.

The aim is to show that the shortest path $\gamma\bigl([0,1]\bigr)$
in~$S_\bmu$ from $U$ to~$V$ lies in~$\Delta$.  (See
Figure~\ref{f:jester-cap}, which also depicts what would happen if
there were two points $\{X_0,X_1\}$ instead of just a single~$X$, as
this case is relevant in Section~\ref{b:mu-lim-log}.)  In particular,
for any $W \in \gamma$, the goal is to show that
\begin{equation}\label{eq:lem:fluctuating-cone_small:2}%\end{equation}
  \dd_s(Z,W) + \dd_s(W,X)
  =
  \pi.
\end{equation}
\begin{figure*}[!ht]
\centering
\includegraphics[width=0.35\textwidth]{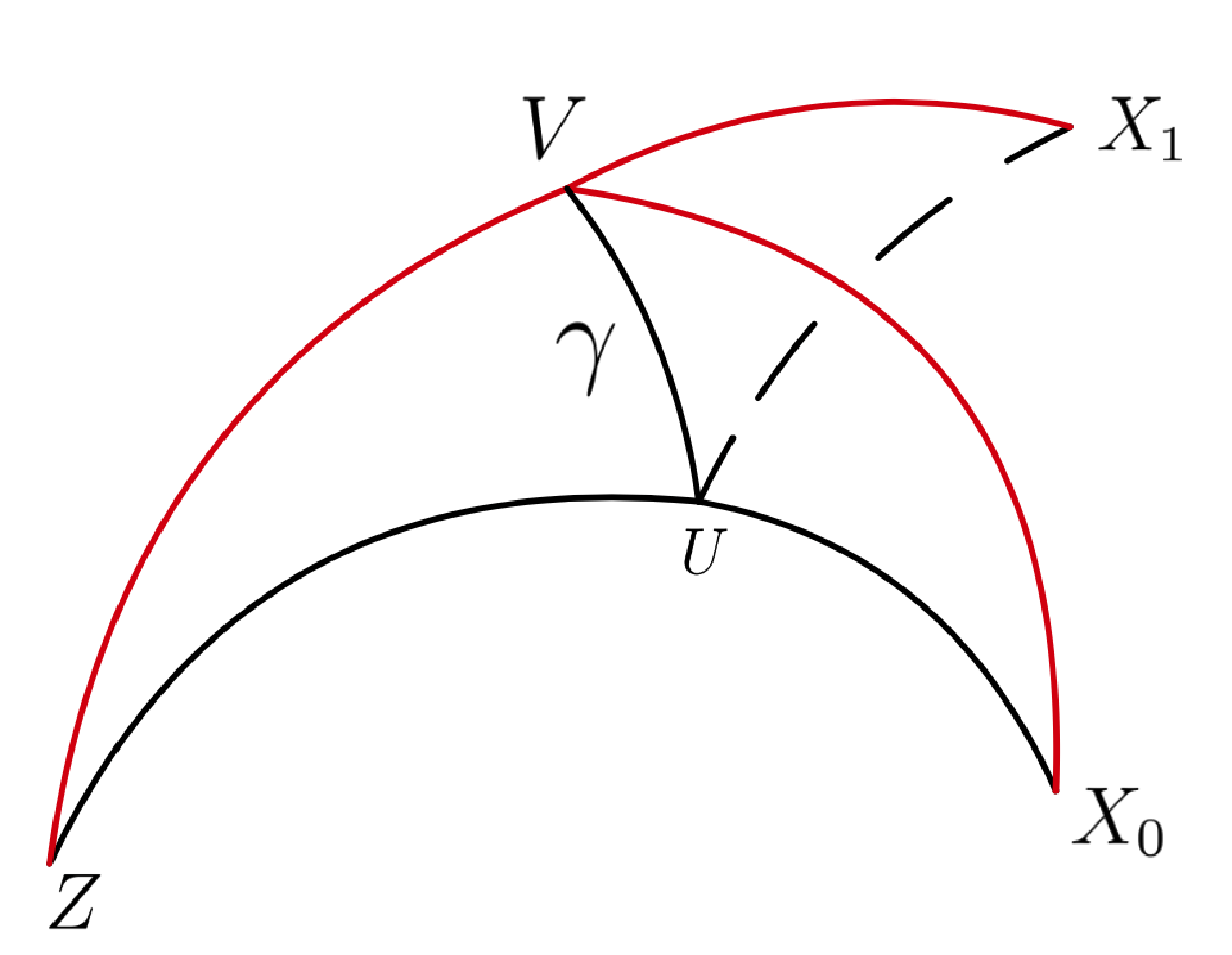}
\caption{Geometry of Lemma~\ref{l:K-is-convex}, if there are two
distinct points~$X_0$ and~$X_1$ instead of just one~$X$.}
\label{f:jester-cap}
\end{figure*}

Consider the two triangles $\triangle ZUV$ and $\triangle UV^{\!}X$
in $S_\bmu$.  Since $\dd_s(U,V) < \pi$,
\begin{align*}%\end{align*}%$$
  \abs{ZU} + \abs{UV} + \abs{ZV} + \abs{UV} + \abs{UX} + \abs{V^{\!}X}
   &= \abs{ZU} + \abs{UX} + \abs{ZV} + \abs{V^{\!}X} + 2\abs{UV}
\\ &= \abs{ZX} + \abs{ZX} + 2\abs{UV}
\\ &< 4\pi.
\end{align*}%$$
So the perimeter of either $\triangle ZUV$ or $\triangle UV^{\!}X$
(or both) is less than~$2\pi$.  Supposing that the perimeter of
$\triangle ZUV$ is less than~$2\pi$, proceed to construct a congruent
triangle $\triangle Z'U'V'$ on the unit sphere.  The other case---when
the perimeter of $\triangle UV^{\!}X$ is less than~$2\pi$---works
similarly.

Construct a congruent model on the Euclidean unit sphere of
dimension~$2$ with $Z'$ at the north pole and $X'$ at the south pole,
where notationally a prime denotes passage to the congruent Euclidean
spherical model.  Let $\sigma_0'$ and $\tau_0'$ be any two half great
circles from pole to pole.  Pick a point~$V'$ on~$\sigma'_0$
satisfying $|V'Z'| = |VZ|$ and a point $U'$ on~$\tau_0'$ satisfying
$|U'Z'| = |UZ|$.  Now slide $\sigma_0'$ closer to or further
from~$\tau_0'$ so that the shortest path $\gamma'$ from~$U'$ to~$V'$
has length $|U'V'| = |UV|$, which is possible because $|UV| < \pi$.
Then $\triangle U'V'Z'$ and $\triangle X'U'V'$ are two triangles in a
space of constant curvature~$1$ congruent to $\triangle UVZ$
and~$\triangle XUV$, respectively.  For any point $W \in \gamma$ pick
a corresponding point $W' \in \gamma'$.  As $(\Smu, d_s)$ is a
$\CAT(1)$ space (that much is true for any locally compact
$\CAT(\kappa)$ space by \cite[Corollary~1.25]{shadow-geom}),
% \ref{c:S-is-CAT(1)}
$$%$$ $
  |Z'W'| \geq |ZW|
  \text{, and }
  |X'W'| \geq |XW|,
$$%$$ $
which implies
$$%$$ $
  \pi
  =
  |Z'W'| + |X'W'|
  \geq
  |ZW| + |XW|
  \geq
  |ZX|
  =
  \pi.
$$%$$ $
Thus equalities must occur throughout, and hence
\begin{equation}\label{eq:fluctuating-cone_small:6}%\end{equation}
  |ZW| + |WX| = |ZX| = \pi.
\end{equation}

Therefore \eqref{eq:lem:fluctuating-cone_small:2} is proved, so $W \in
\Delta$ as desired.
\end{proof}

\begin{proof}[Proof of Proposition~\ref{p:fluctuating-cone_small}]
Suppose, contrary to the conclusion of the Proposition, that distinct
unit vectors $X_0,X_1 \in \Cmu$ have angles with~$Z$ both
% weakly exceeding~$\pi$.
$\geq \pi$.  The goal is to conclude that one of $X_0$ and~$X_1$ does
not lie in the convex cone $\hull\mu$ generated by the support
of~$\hmu = (\log_\bmu)_\sharp\mu$.  In fact the argument shows that
neither~$X_0$ nor~$X_1$ lies in~$\hull\mu$.

Lemma~\ref{l:inj-of_limit-log-on_Cmu_1} implies that $\angle(Z,X_0) =
\pi = \angle(Z,X_1)$ and that for $\mu$-almost all~$x \in \MM$,
$$%$$ $
  \angle(\log_\bmu x,Z) + \angle(\log_\bmu x,X_0)
  =
  \pi
  =
  \angle(\log_\bmu x,Z) + \angle(\log_\bmu x,X_1).
$$%$$ $
Lemma~\ref{l:K-is-convex}, applied to $X_i$ for $i \in \{0,1\}$,
yields convex sets~$\KK_i \subseteq \Tmu$ with
$$%$$ $
  \KK_i
  =
  \bigl\{Y \in \Tmu \mid
    \angle(Y,Z) + \angle(Y,X_i)
    =
    \pi
  \bigr\}.
$$%$$ $
Thus the convex cone $\hull\mu$ is contained in the convex set
$$%$$ $
  \KK = \KK_0 \cap \KK_1.
$$%$$ $
Now simply note that $X_0$ and~$X_1$ are not in~$\KK$ as $X_0 \not\in
\KK_1$ and $X_1 \not\in \KK_0$.  Thus neither $X_0$ nor~$X_1$ lies in
$\hull\mu$.
\end{proof}

\begin{cor}\label{c:inj-of-limit-log-on-Cmu}
Fix a localized measure $\mu$ on a $\CAT(\kappa)$ metric space.  For
any escape vector $Z \in \Emu$ the restriction $\LL_Z|_{\Cmu}: \Cmu
\to\nolinebreak \LL_Z(\Cmu)$ of the limit log map along~$Z$ to the
fluctuating cone is an isometry onto its image.
\end{cor}
\begin{proof}
Let $\KK = \Cmu$ in Theorem~\ref{t:isometry-limit-log} via
Proposition~\ref{p:fluctuating-cone_small}.
\end{proof}

\begin{remark}\label{r:miracle}
The isometry in Corollary~\ref{c:inj-of-limit-log-on-Cmu} is the
miracle that empowers tangential collapse (Section~\ref{s:collapse})
to relate CLTs in singular settings \cite{escape-vectors} to
recognizable linear CLTs.  This isometry is the reason to define
fluctuating cones~$\Cmu$ as distinct from escape cones~$\Emu$, because
Corollary~\ref{c:inj-of-limit-log-on-Cmu} fails for~$\Emu$, in
general: the Fr\'echet mean of a measure supported on the spine of an
open book \cite[Theorem~2.9]{hotz-et-al.2013} can be nudged onto any
desired page of the open book by adding a mass on that page, so the
escape cone is the entire open book, which collapses under limit log
along any individual page.
\end{remark}

\begin{remark}\label{r:inj-of-limit-log-on-Cmu}
For collapse by iterative d\'evissage in Section~\ref{s:collapse},
assuming the isometry in Corollary~\ref{c:inj-of-limit-log-on-Cmu}
would suffice instead of the stronger localized hypothesis on~$\mu$.
What the localized hypothesis adds is convexity of the Fr\'echet
function (Proposition~\ref{p:convexity-of_Df}).  That convexity is
used for estimates in subsequent work \cite{escape-vectors}; but
again, localized might be too strong a hypothesis for optimal
generality in that regard, as well.
\end{remark}

% \comment{subsubsection ``Escape stratum
% containment''~\ref{ss:escape-contain} was here, but it relies on def
% of ``smoothly stratified'' so it had to be moved to the end of
% {\rm\S}\ref{s:strat-spaces}}
% 
% \subsubsection{Escape stratum containment}\label{ss:escape-contain}%%%
% \mbox{}\medskip

%%%%%%%%%%%%%%%%%%%%%%%%%%%%%%%%%%%%%%%%%%%%%%%%%%%%%%%%%%%%%%%%%%%%%%
\section{Smoothly stratified metric spaces}\label{s:strat-spaces}%%%%%
%%%%%%%%%%%%%%%%%%%%%%%%%%%%%%%%%%%%%%%%%%%%%%%%%%%%%%%%%%%%%%%%%%%%%%

\begin{defn}[Smoothly stratified metric space]\label{d:stratified-space}
A complete, geodesic, locally compact $\CAT(\kappa)$ space $(\MM,\dd)$
is a \emph{smoothly stratified metric space} if it decomposes
$$%$$ $
  \MM = \bigsqcup_{j=0}^d \MM^j
$$%$$ $
into disjoint locally closed \emph{strata}~$\MM^j$ so that for
each~$j$, the stratum~$\MM^j$ has closure
$$%$$ $
  \overline{\MM^j} = \bigcup_{k\leq j}\MM^k,
$$%$$ $
and the following conditions hold.
\begin{enumerate}
\item\label{i:manifold-strata}%
(Manifold strata).  For each stratum $\MM^j$, the space
$(\MM^j,\dd|_{\MM^j})$ is a smooth manifold with geodesic distance
$\dd|_{\MM^j}$ that is the restriction of~$\dd$ to~$\MM^j$.

\item\label{i:exponentiable}%
(Local exponential maps).  For any point $p \in \MM$, there exists
$\ve >0$ such that the restriction of the logarithm map $\log_p$ to an
open ball $B(p,\ve) \subseteq \MM$ is a homeomorphism, so the
exponential map $\exp_p = \log_p^{-1}$ is locally well defined and a
homeomorphism.

\begin{excise}{%
  \item\label{i:equisingularity}%
  (Stratum equisingularity).  
  % Survey of stratified spaces
  % Robert Macpherson (May 22, 2012)
  % was
  % http://beta.mbi.ohio-state.edu/video/player/
  % ?id=1090&title=Survey+of+stratified+spaces%20title=
  % but is now
  % https://video.mbi.ohio-state.edu/video/player/?id=1090
  % 32min, 19sec:
  For any pair of points $p$ and~$q$ lying in the same stratum, there is
  a homeomorphism $\MM \to \MM$ that takes $p$ to~$q$.
}\end{excise}%
\end{enumerate}
\end{defn}

\begin{remark}\label{r:space-of_direction_complete}
Local compactness implies that the space of directions is compact
because it is homeomorphic to a compact sphere of small radius around
the apex by Definition~\ref{d:stratified-space}.\ref{i:exponentiable}.
\end{remark}

\begin{remark}\label{r:far-away}
For our purposes, the exponential map is only required to behave well
near the Fr\'echet mean.  Thus the main results of this paper continue
to hold if the sample space has bad points---where the singularities
get worse---outside a ball of positive radius around the Fr\'echet
mean~$\bmu$.
\end{remark}

\begin{remark}\label{r:local-triviality}
Axioms for stratified spaces often include hypotheses explicitly or
implicitly designed to force local triviality of tangent data within
each stratum.  That role is played here by local exponential maps in
Definition~\ref{d:stratified-space}.\ref{i:exponentiable}.  Local
triviality is not needed for the developments here or in the sequels,
\cite{random-tangent-fields} and~\cite{escape-vectors}, so a proof of
it (based on radial transport \cite[Definition~2.7
% \ref{d:radial-transport} and \ref{p:radial-is-isometry}
% and \mbox{Proposition}~2.8]{shadow-geom})~is~not~\mbox{included}.
and Proposition~2.8]{shadow-geom})~is~omitted.
\end{remark}

\begin{remark}\label{r:not-riemannian}
Our definition of smoothly stratified metric space does not require a
Riemannian structure on each stratum.  That is because angles---and
thus inner products---on the tangent bundle are not required to vary
smoothly.  Indeed, angles on a smoothly stratified metric space can be
discontinuous (see \cite[Remark~1.22]{shadow-geom}).
% \ref{r:angle-not-continuous}
% In lieu of a Riemannian structure, central limit theorems on
% smoothly stratified metric spaces require the distance function
% on~$\MM$ to have directional derivatives of order~$2$ (this is the
% amenable hypothesis in Definition\verb=~\ref{d:amenable}= from
% \cite{escape-vectors}).
Potential alternatives to the existence of local exponential maps in
Definition~\ref{d:stratified-space}.\ref{i:exponentiable} include
Pflaum's Riemannian stratified spaces \cite{pflaum2001} or iterated edge
metrics of Albin, Leichtnam, Mazzeo, and Piazza \cite{ALMP12}, but
issues arise with these.
\begin{enumerate}
\item%
In Pflaum's definition, a vector flow at a singular point $p$ cannot
leave the singular stratum containing~$p$, so control is relinquished
over the behavior of angles and geodesics near singular points.
\item%
Similarly, \cite{bertrand-ketterer-mondello-richard2018} note that
local behavior of geodesics near singular points are not well
understood when $\MM$ is equipped with an iterated edge metric.
\item%
It might be possible to replace
Definition~\ref{d:stratified-space}.\ref{i:manifold-strata} with a
requirement that the metric $\dd$ induces a Riemannian metric on each
stratum~$\MM^j$.  However, similar to Pflaum's Riemannian stratified
spaces~\cite{pflaum2001}, this does not guarantee control of geodesics near
singular points.
\end{enumerate}
Weakenings of the Riemannian framework on manifolds, such as connection
spaces \cite{pennec-lorenzi2020}, which still retain enough structure to
take derivatives
% but not Levi--Civita
and exponentials could be useful in stratified settings, but for the
present purpose remain too strong, as limit log maps
(Definition~\ref{d:limit-log}) only require parallel transport in
radial directions (see \cite[Section~2]{shadow-geom}
% \ref{s:radial}
for details about radial transport).
\end{remark}

\begin{figure*}[!ht]
\label{f:parabola_1}
\centering
\includegraphics[width=0.35\textwidth]{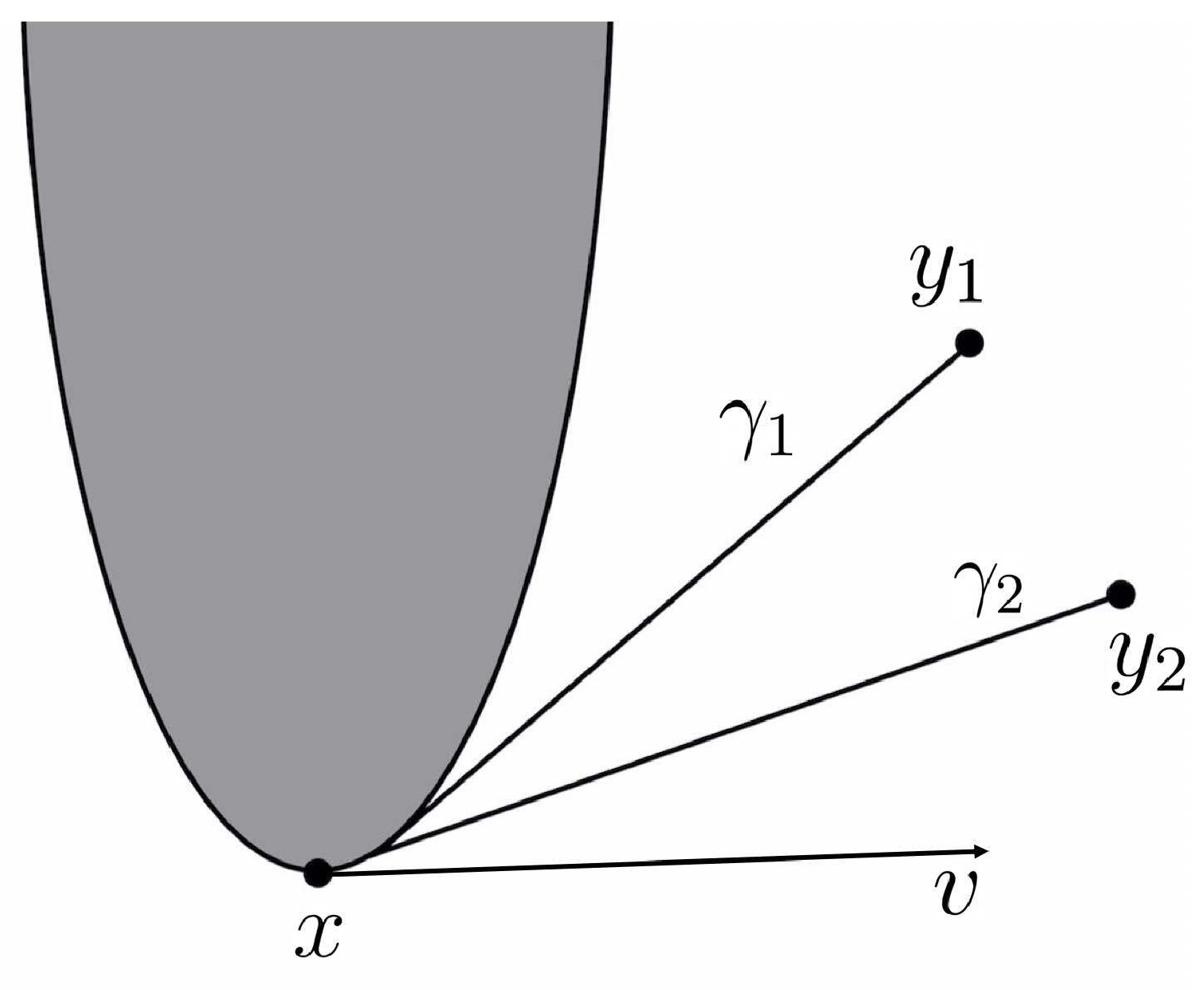}
\caption{$\exp_x$ is not well defined}
\end{figure*}

\begin{remark}\label{r:tubular}
Definition~\ref{d:stratified-space}.\ref{i:exponentiable} should be
compared with existence of maximal tubular neighborhoods on stratified
spaces via the exponential map on stratified spaces studied by Pflaum
\cite{pflaum2001}.  Example~\ref{e:not_mfd_with_corner} is a scenario
when the maximal tubular neighborhood of the parabola stratum $\MM^0$
exists while Definition~\ref{d:stratified-space}.\ref{i:exponentiable}
is not satisfied.
\end{remark}

\begin{example}\label{e:not_mfd_with_corner}
Let $\MM = \{(x,y) \in \RR^2 \mid y \leq x^2\}$
(Figure~\ref{f:parabola_1}) be the space with strata
\begin{itemize}
\item%
$\MM^0 = \{(x,y) \in \RR^2 \mid y = x^2\}$
\item%
$\MM^1 = \MM \setminus M_0$
\end{itemize}
whose metric on~$\MM^1$ is the induced flat metric from $\RR^2$ and
whose metric on~$\MM^2$ is the curve length metric on the parabola $y
= x^2$.  For a piecewise smooth curve $\gamma$ in~$\MM$, the length
of~$\gamma$ is the sum of the length of the intersection with~$\MM^0$
and~$\MM^1$,
$$%$$ $
  |\gamma| = |\gamma \cup \MM^0| + |\gamma \cap \MM^1|.
$$%$$ $
The metric on~$\MM$ is defined as
$$%$$ $
  \dd(x,y) = \inf_{\begin{subarray}{c} \gamma:[0,1] \to \MM\ \\
  \gamma[0]=x,\ \gamma[1]=y \end{subarray}} |\gamma|.
$$%$$ $
Then $(\MM,\dd)$ is a length space and also a Riemannian stratified
space in the sense of Pflaum \cite[Definition~2.4.1]{pflaum2001}.  But
it is not a smoothly stratified metric space by
Definition~\ref{d:stratified-space} since the exponential map at~$x$
is not locally defined.

It bears mentioning that this example reflects a geometric problem and
not a topological one: the fact that $\MM$ has nonempty boundary is
not the issue.  Indeed, the discussion would still work after taking
the union of~$\MM$ with a half-infinite cylinder with the parabola as
its base, the union being homeomorphic to the plane~$\RR^2$.
\end{example}

\begin{remark}\label{r:codim}
The tangent cone at a singular point~$p$ of a stratified space is
often thought of as a product of the vectors tangent to the stratum
containing~$p$ and a normal cone of vectors pointing out of the
stratum.  The normal cone is a cone over its unit vectors, which form
the \emph{link} of the singular point.  The topological type of the
link reflects the nature of the singularity type (see
\cite[p.\,7]{goresky-macpherson1988}, for example), with smooth points
having links that are spheres, mildly singular points having smooth
links, and deeper singularities having links that are one step less
singular.  In this metric setting, however, even a point whose link is
a sphere can be singular; that occurs, for example, in the kale
\cite{kale-2015}, which is homeomorphic to a Euclidean plane but not
isometric to it, having an isolated point of angle sum $> 2\pi$.
Since part of our theory collapses by inductively reducing the
``depth'' of a singularity at~$p$ via limits of tangent cones upon
approach to~$p$
% (Sections~\ref{b:limit-log}~and~\ref{s:collapse}),
(Section~\ref{s:collapse}), a proxy is needed to bound the ``depth''
of the singularity to guarantee that this
% d\'evissage
procedure terminates.
% but it isn't important to nail the ``depth'' on the head; all we
% need is an upper bound
For that purpose, the codimension of the stratum containing~$p$
suffices, as follows, since the codimension increases by at least~$1$
with each iterated cone over a link.  Definition~\ref{d:codim} is used
only in Proposition~\ref{p:ppt-L-v}.
\end{remark}

\begin{defn}\label{d:codim}
A point $p$ in a smoothly stratified metric space~$\MM$ has
\emph{codimension}
$$%$$ $
  \codim p = M - m,
$$%$$ $
where $m$ is the dimension of the (unique smallest) stratum
containing~$p$ and $M$~is the maximal dimension of any stratum whose
closure contains~$p$.
\end{defn}

\begin{remark}\label{r:less-singular}
Points of lower codimension should be thought of as less singular than
points of higher codimension.  Although the set of singular points is
larger when the codimension is smaller, that set is just a stratum's
worth of copies of the same singularity.  The type of the singularity
is recorded more faithfully by the normal cone, whose dimension
decreases along with the codimension.
\end{remark}

This subsection concludes with crucial components of tangential
collapse in Section~\ref{s:collapse}, extensions of the geometry in
Section~\ref{s:prereqs} available in the smoothly stratified setting.
The first reduces from the setting of smoothly stratified metric
spaces, where curvature is bounded above by~$\kappa$, to nonpositively
curved tangent cones, where curvature is bounded above by~$0$.  It is
applied in the proof of Proposition~\ref{p:ppt-L-v}.

\begin{prop}\label{p:conical-metric}
The conical metric on~$\TpM$ makes the tangent cone of any point~$p$
in a smoothly stratified metric space~$\MM$ a smoothly stratified
nonpositively curved~space.
\end{prop}
\begin{proof}
The space is nonpositively curved by
Lemma~\ref{l:limit-log-preserve-npc}.  The stratification decomposes
into disjoint strata appropriately for
Definition~\ref{d:stratified-space} because~$\MM$ itself does and
local exponentiation is a homeomorphism in
Definition~\ref{d:stratified-space}.\ref{i:exponentiable}.
% Every stratum that is not the apex is the cone over a stratum
% of~$\SpM$.
The metric on $\TpM$ is induced (via angles in
Definition~\ref{d:angle} and Lemma~\ref{l:angular-metric}) by the
metric on~$\MM$ itself, which restricts appropriately to its strata by
Definition~\ref{d:stratified-space}.\ref{i:exponentiable}.  Therefore
the metric on~$\TpM$ restricts appropriately for
Definition~\ref{d:stratified-space}.\ref{i:manifold-strata} because of
how the conical metric on~$\TpM$ is constructed in
Definitions~\ref{d:tangent-cone} and~\ref{d:conical-metric}.
\end{proof}

The next result is applied in the last step of tangential collapse
(Proposition~\ref{p:ppt-LL-confined}).

\begin{lemma}\label{l:decrease-inner-product}
In the situation of Proposition~\ref{p:conical-metric}, geodesic
projection of~$\TpM$ onto the tangent vector space at~$p$ of the
stratum~$R$ containing~$p$, namely
\begin{align*}%\end{align*}%$$
   \bP: \TpM & \to T_p R
\\
             W &\mapsto \argmin_{Z\in T_p R} \dd_p(Z,W)
\end{align*}%$$
is a contraction, in the sense that it
weakly decreases angles: for
$Z \!\in T_p R$ and\/~$V \!\in\nolinebreak \TpM$,
$$%$$ $
  \<\bP(V),Z\>_p
  \geq
  \<V,Z\>_p.
$$%$$ $
\end{lemma}
\begin{proof}
For readability, set $W = \bP(V)$, write $\OO$ for the apex of the
tangent cone $T_p R \subseteq \TpM$, and omit the subscripts on
pairings.  Assume $W \neq \OO$ and $V \notin T_\OO R$ is not already
tangent to~$R$, since otherwise $W = V$ and equality on the pairings
is trivial.  Note that $\angle(V,W) < \pi/2$ because otherwise
$\dd_\OO (V,\OO) \leq \dd_\OO(V,W)$, which violates the definition
of~$\bP(V)$.

In what follows, for $U,V \in T_\OO\MM$ write $UV$ to mean the
shortest path from $U$ to~$V$ and $|UV|$ to mean the length
$\dd_\OO(U,V)$ of this path.
\begin{figure*}[!ht]\label{f:triangle}
\centering
\includegraphics[scale =0.25]{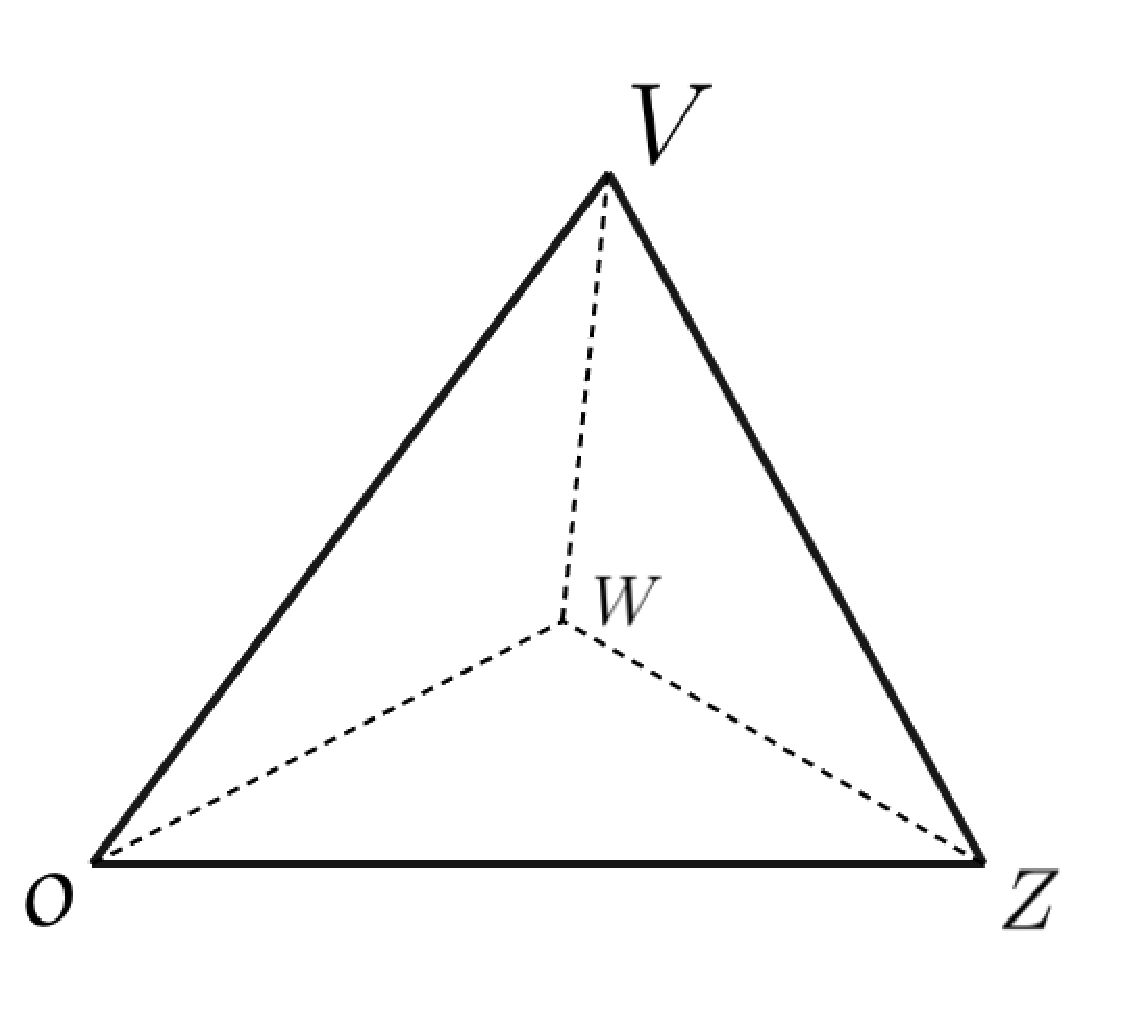}
\caption{}
\end{figure*}
The triangle $\triangle VWZ$ in the $\CAT(0)$ space $T_\OO\MM$ has
$\angle(WV,WZ)\geq \pi/2$ by \cite[Proposition~2.4]{bridson2013metric}
because $W$ is the projection of~$V$, so
\begin{equation}\label{eq:lem:P_K_reduce_inner_prod:1}%\end{equation}
  |VZ|^2 \geq |VW|^2+|WZ|^2.
\end{equation}
Invoke Lemma~\ref{l:flat-triangle} to see that the triangle $\triangle
V\OO Z$ is flat.  Rescaling if necessary, assume that $V$ and~$Z$ are
unit vectors.  The law of cosines applied to them yields
\begin{equation}\label{eq:lem:P_K_reduce_inner_prod:2}%\end{equation}
  |VZ|^2 = 2 - 2\<V,Z\>.
\end{equation}
Combing \eqref{eq:lem:P_K_reduce_inner_prod:1}
and~\eqref{eq:lem:P_K_reduce_inner_prod:2} produces
\begin{equation}\label{eq:lem:P_K_reduce_inner_prod:3}%\end{equation}
  2\<V,Z\> \leq 2 - |VW|^2 - |WZ|^2.
\end{equation}
Now observe that the triangle $\triangle V\OO W$ is also flat (by
applying Lemma~\ref{l:flat-triangle} again) and $\angle(WV,W\OO) =
\pi/2$, so
\begin{equation}\label{eq:lem:P_K_reduce_inner_prod:4}%\end{equation}
  1 = |\OO V|^2 = |\OO W|^2 + |VW|^2.
\end{equation}
{}From \eqref{eq:lem:P_K_reduce_inner_prod:3}
and~\eqref{eq:lem:P_K_reduce_inner_prod:4} we obtain
\begin{equation}\label{eq:lem:P_K_reduce_inner_prod:5}%\end{equation}
  2\<V,Z\> \leq 1 + |\OO W|^2 - |WZ|^2
  =
  |\OO Z|^2 + |\OO W|^2 - |WZ|^2
  =
  2\<W,Z\>
\end{equation}
where the last equality comes from the law of cosines in the triangle
$\triangle \OO WZ$, which is again flat by Lemma~\ref{l:flat-triangle}.
Thus $\<V,Z\> \leq \<W,Z\>$, as desired.
\end{proof}

% \subsubsection{Escape stratum containment}\label{ss:escape-contain}%
% \mbox{}\medskip
% 
% \noindent

The final result in this section is used in the proof of
Proposition~\ref{p:ppt-LL-confined} on the way to existence of
tangential collapse in Theorem~\ref{t:collapse}.  It says that after
taking limit log, the escape cone always includes the stratum
containing the mean.

\begin{prop}\label{p:nicer-fluctuating-cone-after-limit-log}
Under Hypotheses~\ref{h:XX}, assume in addition that~$\MM$ is smoothly
stratified.  Suppose $\LL_Z(Z)$ lies in a stratum $\TZR \subseteq
\TZM$, where $R$ is the stratum of~$\MM$ toward which $Z$ points.
Then the
% fluctuating cone $C_{\mu_Z}$ of $\mu_Z$.  [this appears to
% be false with the $\hull$ def of~$\Cmu$, because the $\hull$ can be
% a proper vector subspace; state instead in terms of~$\Emu$]
escape cone $\Emuz$ of~$\mu_Z$ contains the vector space~$\TZR$.
\end{prop}
\begin{proof}
The goal
% [$\hull$ condition is needed to prove the stated assertion; but
% perhaps what's actually needed is $Z \in \Emu \implies \TZR
% \subseteq \Emuz$]
is to show that, for all $V_Z \in \TZR$,
$$%$$ $
  \nabla_{\!\bmu_Z} F_{\mu_Z}(V_Z) = 0.
$$%$$ $
Recall from~\eqref{eq:limit-log_preserve_mean:1} that
$$%$$ $
  \nablabmu F_{\mu_Z}(\LL_Z Z)
  =
  \nablabmu F_\hmu(Z)
  =
  0.
$$%$$ $
Applying Lemma~\ref{l:sum=pi} and Corollary~\ref{c:nablamu(F)} yields
\begin{align*}%\end{align*}%$$
  -\nabla_{\!\bmu_Z}F_{\mu_Z}(\LL_Z Z)
  &= \int_{\TZM}\<\LL_Z Z,W_z\> \mu_Z(dW_z)
\\*[.5ex]
  &= -\int_{\TZM}\<-\LL_Z Z,W_z\> \mu_Z(dW_z)
\\*[.5ex]
  &= \nablabmu F_{\mu_Z}(-\LL_Z Z).
\end{align*}%$$
Consequently,
$$%$$ $
  \nabla_{\!\bmu_Z} F_{\mu_Z}(-\LL_Z Z)
  =
  \nabla_{\!\bmu_Z} F_{\mu_Z}(\LL_Z Z)
  =
  0.
$$%$$ $
Combine this with the fact that $\bmu_Z$ minimizes the convex
function~$F_{\mu_Z}$, so the directional derivative
$\nabla_{\!\bmu_Z}F_{\mu_Z}$ is convex and nonnegative on the Euclidean
space $\TZR$,~to~get
\begin{equation*}%\end{equation*}%$$
  \nabla_{\!\bmu_Z}F_{\mu_Z}(V_Z) = 0 \text{ for all } V_Z \in \TZR.
  \qedhere
\end{equation*}%$$
\end{proof}

%%%%%%%%%%%%%%%%%%%%%%%%%%%%%%%%%%%%%%%%%%%%%%%%%%%%%%%%%%%%%%%%%%%%%%
\section{Tangential collapse}\label{s:collapse}%%%%%%%%%%%%%%%%%%%%%%%
%%%%%%%%%%%%%%%%%%%%%%%%%%%%%%%%%%%%%%%%%%%%%%%%%%%%%%%%%%%%%%%%%%%%%%

\noindent
The goal is to embed the fluctuating cone~$\Cmu$ into a Euclidean
space by collapsing the local singularity at the Fr\'echet mean
(Theorem~\ref{t:collapse}).  This is done by a series of
\emph{d\'evissage steps}, each taking a limit logarithm along a
resolving direction.
% (Definition~\ref{d:resolving-direction}).
Each d\'evissage (Definition~\ref{d:step-devissage}) pushes forward the
measure $\mu$ under the limit log along a resolving direction.  The
result is a measure on a stratified space whose singularity has lower
codimension (Definition~\ref{d:codim})
% [also Remark~\ref{r:less-singular} in clt-on-strat221011.tex
% has relevant text]
and is thus less singular.

Results throughout this section specify certain hypotheses, so it is
important to recall these, especially localization
(Definition~\ref{d:localized}) for measures on smoothly stratified
metric spaces (Definition~\ref{d:stratified-space}).  Note also that
the collapse by iterative d\'evissage heavily involves limit log maps
(Definition~\ref{d:limit-log}) along resolving directions
(Definition~\ref{d:resolving-direction}).

%%%%%%%%%%%%%%%%%%%%%%%%%%%%%%%%%%%%%%%%%%%%%%%%%%%%%%%%%%%%%%%%%%%%%%
\subsection{D\'evissage}\label{b:dévissage}%%%%%%%%%%%%%%%%%%%%%%%%%%%
%mbox{}\medskip%%%%%%%%%%%%%%%%%%%%%%%%%%%%%%%%%%%%%%%%%%%%%%%%%%%%%%%

\begin{prop}\label{p:dévissage-initial-step}
Fix a localized
% amenable
% [doesn't need immured here---or probably anywhere else in
% {\rm\S}\ref{s:collapse} for that matter, as the immured hypothesis
% is there solely to ensure confinement of escape vectors to~$\Cmu$]
measure~$\mu$ on a smoothly stratified metric space~$\MM$.  For any
unit vector $Z$ in the escape cone~$\Emu$
(Definition~\ref{d:escape-cone}) let~$\mu_Z = (\LL_Z)_\sharp\hmu$ be
the pushforward of $\hmu = (\log_\bmu)_\sharp\mu$ under limit log map
along~$Z$.  Then
\begin{enumerate}
\item\label{i:mean}%
$\bmu_Z = \LL_Z\bmu$ is the Fr\'echet mean of $\mu_Z$;
\item\label{i:LL_Z-isometry}%
$\LL_Z$ maps the fluctuating cone $\Cmu$ isometrically to its image,
so for all $V,W \in \Cmu$,
$$%$$ $
  \angle\bigl(\LL_Z(V),\LL_Z(W)\bigr) = \angle(V,W);
$$%$$ $
\item\label{i:Cmuz}%
the fluctuating cone~$\Cmuz$ of~$\mu_Z$ satisfies $\LL_Z(\Cmu)
\subseteq \Cmuz$;
\item\label{i:muZ-localized}%
the measure $\mu_Z$ is localized.
\end{enumerate}
\end{prop}
\begin{proof}
The first claim comes from Proposition~\ref{p:limit-log-preserve-mean}
(see also Lemma~\ref{l:mean-of-hmu}), and the second one comes from
Corollary~\ref{c:inj-of-limit-log-on-Cmu}, both of which use the
localized hypothesis.  The third is
Corollary~\ref{c:LL-preserves-Cmu}.  For the fourth, $\mu_Z$ is
automatically localized by Example~\ref{e:localized} because by
Lemma~\ref{l:limit-log-preserve-npc} the limit tangent cone $\TZM$
is~$\CAT(0)$.
\end{proof}

\begin{defn}\label{d:resolving-direction}
A~\emph{resolving direction} for a measure~$\mu$ on a smoothly
stratified~$\CAT(\kappa)$ metric space is a unit vector in the
fluctuating cone~$\Cmu$ pointing to a stratum of highest dimension
among strata having nonempty intersection with the exponential image
$$%$$ $
  \{\exp_\bmu V \mid V \in \Cmu \text{ can be exponentiated}\}
$$%$$ $
of the fluctuating cone as per
Definition~\ref{d:stratified-space}.\ref{i:exponentiable}.
\end{defn}

\begin{remark}\label{r:devissage}
Proposition~\ref{p:dévissage-initial-step} implies that applying the
limit log map along a resolving direction results in a stratified
space with the Fr\'echet mean~$\bmu$ and the fluctuating cone~$\Cmu$
preserved.  It turns out that~$\bmu$ and~$\Cmu$ is all of the
information needed to reproduce the limiting distribution for the CLT
in a less singular stratified space; see the Perturbative CLT in
\cite[Section~6.2]{escape-vectors}.
% \ref{t:perturbation-CLT}
And it is the essence of d\'evissage in
Definitions~\ref{d:step-devissage} and~\ref{d:dévissage}: collapse by
iterating the limit log map in
Proposition~\ref{p:dévissage-initial-step}.  Conditions for this
iteration occupy Definition~\ref{d:ravel}.
\end{remark}

\begin{defn}\label{d:ravel}
A \emph{ravel}
% alternative: \emph{crumple}
is a triple $\{\tM, \nu, V\}$ where $\tM$ is a smoothly stratified
metric space, $\nu$~is a probability measure on~$\tM$, and $V$~is a
resolving direction of~$\nu$.
\end{defn}

\begin{defn}\label{d:step-devissage}
Given a ravel $\{\tM, \nu, V\}$, a \emph{d\'evissage step} takes a
limit log along the direction~$V$ to transfer the ravel $\{\tM, \nu,
V\}$ to a new one $\{\tM', \nu\hspace{.25ex}', V'\}$, where
\begin{enumerate}
\item%
$\tM' = \LL_{V}(\tM)$,
\item%
$\nu\hspace{.25ex}' = (\LL_V)_\sharp\nu$, and
\item%
$V'$ is a resolving direction of~$\nu\hspace{.25ex}'$.
\end{enumerate}
\end{defn}

\begin{defn}\label{d:resolved}
Fix a ravel $\{\tM, \nu, V\}$ whose Fr\'echet mean~$\bnu$ lies in a
stratum~$R$.  The ravel $\{\tM, \nu, V\}$ is \emph{resolved} if
$\fc_\nu \subseteq T_{\bar\nu} R$.  Otherwise, $\{\tM, \nu, V\}$ is
\emph{unresolved}.
\end{defn}

\begin{remark}\label{r:resolved}
When a ravel $\{\tM, \nu, V\}$ is unresolved, the local geometry can
be improved, in the sense that the singularity codimension can be
reduced by taking the limit log map along a resolving direction, as in
Proposition~\ref{p:dévissage-initial-step}.  When $\{\tM, \nu,
V\}$ is resolved, however, the limit log map is constrained to take
place within the same stratum as the Fr\'echet mean~$\bnu$, so the
limit tangent cone at~$\bnu$ has the same singularity as the tangent
cone at~$\bnu$ itself.  The process of d\'evissage here, made precise
in Definition~\ref{d:dévissage}, is a sequence of d\'evissage
steps (thought of as ``unravelings'') applied to unresolved ravels
until a resolved ravel is achieved.
\end{remark}

The starting ravel of interest is the tangent cone $\{\Tmu, \hmu,
Z\}$, where $\hmu = (\log_\bmu)_\sharp\mu$ is pushed forward
from~$\MM$ to the tangent cone under the log map, as usual.

\begin{defn}\label{d:dévissage}
Fix a measure~$\mu$ on a smoothly stratified metric space~$\MM$.  The
process of \emph{d\'evissage} is a sequence of d\'evissage steps,
beginning with an initial ravel $\{\MM_0, \nu_0, Z_0\} = \{\Tmu, \hmu,
Z\}$ for any resolving direction $Z$ of~$\mu$ and whose step~$i$
applies a d\'evissage step to the ravel $\{\MM_i, \nu_i, Z_i\}$ to
produce the ravel $\{\MM_{i+1}, \nu_{i+1}, Z_{i+1}\}$, where
\begin{enumerate}
\item%
$\MM_{i+1} = \vT_{Z_i}(\MM_i)$,
\item%
$\nu_{i+1} = (\LL_{Z_i})_\sharp\nu_i$, and
\item%
$Z_{i+1}$ is a resolving direction of~$\nu_{i+1}$.
\end{enumerate}
The d\'evissage \emph{terminates} at step~$k$ when the ravel
$\{\MM_{k+1}, \nu_{k+1}, Z_{k+1}\}$ is resolved
(Definition~\ref{d:resolved}).  To simplify notation in discussions of
d\'evissage, let
\begin{itemize}
\item%
$\LL_i = \LL_{Z_i}$ be the limit logarithm map along direction~$Z_i$,
\item%
$\fc_i = \Cni{i}$ be the fluctuating cone of $\nu_i$, and
\item%
$\dd_i = \dd_{\bnu_i}$ be the conical distance function on~$\MM_i$
(Definition~\ref{d:conical-metric}).
\end{itemize}
\pagebreak[2]
If termination occurs at step~$k$, then write
\begin{itemize}
\item%
$\MM_\infty = \MM_{k+1}$ for the \emph{terminal} smoothly stratified
metric space,
\item%
$\nu_\infty = \nu_{k+1}$ for the \emph{terminal} measure
on~$\MM_\infty$,
\item%
$\dd_\infty = \dd_{\bnui}$ for the \emph{terminal} distance
function on~$\MM_\infty$,
\item%
$R_\infty$ for the \emph{terminal} stratum of~$\MM_\infty$ that
contains~$\bnui$,
\item%
$\Ti = T_{\bnui}R_\infty \cong \RR^{\dim R_\infty}$ for the
\emph{terminal tangent vector space},
\item%
$\Ci = \Cni{\infty} \subseteq \Ti$ for the \emph{terminal} fluctuating
cone, and
\item%
$Z_\infty = Z_k \in \Ti$ for the \emph{terminal} resolving direction.
\end{itemize}
\end{defn}

\begin{remark}\label{r:sticky}
D\'evissage terminates when resolving directions fail to point out of
the current Fr\'echet mean stratum, so further limit logarithms along
resolving directions do not reduce the singularity codimension.  That
happens, for example, when the last d\'evissage moves to a top
stratum, i.e.,~with codimension~$0$.  But it is possible for
d\'evissage to stick to a stratum of positive codimension.  Such is
the case when $\MM$ is an open book \cite{hotz-et-al.2013}, where the
Fr\'echet mean is confined to the spine.
\end{remark}

\begin{lemma}\label{l:codim-decrease}
In any d\'evissage from Definition~\ref{d:dévissage}, $\LL_i(\bnu_i) =
\bnu_{i+1}$ and the singularity at~$\bnu_{i+1}$ has strictly lower
codimension than~at~$\bnu_i$.
\end{lemma}
\begin{proof}
The Fr\'echet mean equality is thanks to
Proposition~\ref{p:dévissage-initial-step}.  What remains is to
carry out the dimension calculation showing that d\'evissage
(Definition~\ref{d:step-devissage}) applied to an unresolved ravel $\{\tM,
\nu, V\}$ reduces the codimension of the singularity by at least~$1$
upon passage to $\{\tM', \nu\hspace{.25ex}', V'\}$.  Let~$R$ be the
stratum of~$\MM$ containing~$\bnu$.  By Definition~\ref{d:resolved} of
unresolved, the exponential of~$tV$ for positive $t \ll 1$ lies in a
stratum~$R'$ of~$\MM$ that does not equal~$R$ but whose closure
meets~$R$, as can be seen by letting~$t \to\nolinebreak 0$.  The
stratum of~$\MM'$ containing~$\bnu\hspace{.25ex}'$ contains
$T_{\bnu\hspace{.25ex}'} R'$ and hence has strictly larger dimension
than~$R$ does.  Similarly, any stratum whose closure meets~$R'$ also
contains~$R$ in its closure; taking tangent cones implies that the
maximal dimension of a stratum containing~$\bnu\hspace{.25ex}'$ is at
most the maximal dimension of a stratum whose~closure~contains~$R$.
\end{proof}

\begin{remark}\label{r:infty}
The notations $\MM_\infty$ and $\bnui$ and so on for the terminal
objects are justified because any attempt at d\'evissage beyond
termination would have no effect: the relevant limit log map would be
the identity, or at least a canonical isometry.
\end{remark}

The purpose of d\'evissage is to map the singular situation
surrounding the fluctuating cone of~$\mu$ in~$\MM$ to a smooth
version.  The heavy lifting is carried by composing the d\'evissage
steps.

\begin{defn}\label{d:composite-devissage}
The \emph{composite d\'evissage} is
$$%$$ $
  \LL_\circ
  =
  \LL_k \circ \dots \circ \LL_0:
  \Tmu
  \to
  \vT_{Z_k}\MM_k = \MM_\infty.
$$%$$ $
\end{defn}

Alas, although the target of~$\LL_\circ$ is smooth when the last
d\'evissage step moves to a top stratum, it need not be smooth in
general.  Therefore $\LL_\circ$ in general only produces a partial
tangential collapse, in the following sense.

\begin{defn}[Tangential collapse]\label{d:collapse}
Fix a measure~$\mu$ on a smoothly stratified metric space~$\MM$, and
let~$\TT$ be a conical
% (Definition~\ref{d:MM-is-conical})
(Definition~\ref{d:log-map}) smoothly stratified metric space.  A
\emph{partial tangential collapse} of~$\mu$ with target~$\TT$ is a map
$\LL: \Tmu \to \TT$ such that
\begin{enumerate}
\item\label{i:pushforward-mean}%
$\LL\bigl(\log_\bmu (\bmu)\bigr) = 0$ is the Fr\'echet mean of the
pushforward
$\LL_\sharp\hmu$, where $\hmu = (\log_\bmu)_\sharp\mu$;
\item\label{i:injective}%
$\LL$ is injective on the closure~$\oCmu$ of the fluctuating
cone~$\Cmu \subseteq \Tmu$;
\item\label{i:partial-isometry}%
for any fluctuating vector $U \in \Cmu$ and any tangent vector $V \in
\Tmu$,
$$%$$ $
  \<U,V\>_\bmu = \bigl\<\LL(U),\LL(V)\bigr\>_{\LL(\bmu)}.
$$%$$ $
\item\label{i:homogeneous}%
$\LL$ is \emph{homogeneous}: $\LL(tV) = t\LL(V)$ for all real~$t \geq
0$ and~$V \in \Tmu$; and
\item\label{i:continuous}%
$\LL$ is continuous.
\end{enumerate}
A~\emph{(full) tangential collapse} of~$\mu$
% (i.e., without the ``partial'' modifier)
is a partial collapse $\LL: \Tmu \to\nolinebreak \RR^m$ whose target
is a vector space.
\end{defn}

\begin{remark}\label{r:isometry}
Angle preservation for $V \in \Cmu$ in
Definition~\ref{d:collapse}.\ref{i:partial-isometry} and continuity
means the restriction of~$\LL$ from $\oCmu$ to its image makes the map
$\oCmu \to \LL(\Cmu)$ an~isometry.
\end{remark}

To reach a full tangential collapse, composite
d\'evissage~$\LL_\circ$ must be followed by terminal projection to
squash the singular cone $\vT_{Z_k}\MM_k$ onto the (automatially
Euclidean) tangent space to the stratum containing the terminal
Fr\'echet mean.

\begin{defn}\label{d:terminal-projection}
Given a d\'evissage process, the \emph{terminal projection} is the
geodesic projection of~$T_{\bnui}\MM_\infty$ onto the terminal tangent
vector space~$\Ti = T_{\bnui}R_\infty$, namely
\begin{equation}\label{eq:def_special_projection}%\end{equation}
\begin{split}
   \bPi: T_{\bnui}\MM_\infty & \to \Ti
\\
                           W &\mapsto \argmin_{Z\in\Ti} \dd_\infty(Z,W).
\end{split}
\end{equation}
\end{defn}

\begin{remark}\label{r:trivial-P_k}
Terminal projection~$\bPi$ restricts to the identity on~$\Ti$.
Therefore $(\bPi \circ \LL_\circ)|_{\Cmu}$ is an isometry and $\bPi
\circ \LL_\circ(\bmu)$ is the Fr\'echet mean of $(\bPi \circ
\LL_\circ)_\sharp\hmu$.  When the last d\'evissage step moves to a top
stratum, $\bPi$ is the identity map on~$\Ti$ because~$R_\infty$ is
then a top stratum, so there are no tangent vectors outside
of~$T_{\bnui}R_\infty$; thus $\bPi\circ\LL_\circ = \LL_\circ$ in that
case.  See also Corollary~\ref{c:collapse}, which implies that the
resolution constructed here has slightly stronger properties in that
case.
\end{remark}

\begin{lemma}\label{l:projection-is-continuous}
The terminal projection~$\bPi$ is continuous.
\end{lemma}
\begin{proof}
Convex projection in a $\CAT(0)$ space contracts
\cite[Proposition~II.2.4]{bridson2013metric}.
\end{proof}

%%%%%%%%%%%%%%%%%%%%%%%%%%%%%%%%%%%%%%%%%%%%%%%%%%%%%%%%%%%%%%%%%%%%%%
\subsection{Constructing tangential collapse}\label{b:construction}%%%
\mbox{}\medskip%%%%%%%%%%%%%%%%%%%%%%%%%%%%%%%%%%%%%%%%%%%%%%%%%%%%%%%

\noindent
Throughout this subsection, fix notation as in
Definitions~\ref{d:dévissage} and~\ref{d:terminal-projection}.

\begin{prop}\label{p:ppt-L-v}
Fix a localized measure~$\mu$ on a smoothly stratified metric
space~$\MM$ and a d\'evissage that terminates at step~$k$.  For $i =
0, \dots, k$,
\begin{enumerate}%[a.]
\item\label{i:hypotheses}%
$\nu_i$ is localized on the smoothly stratified $\CAT(0)$
space~$\MM_i$;

\item\label{i:codimension}%
$\LL_i(\bnu_i) = \bnu_{i+1}$ and the singularity at~$\bnu_{i+1}$ has
strictly lower codimension than~at~$\bnu_i$;

\item\label{i:isometry}%
$\LL_i$ isometrically maps $\fc_i$ to $\LL_i(\fc_i)$;

\item\label{i:fluctuating-cone}%
the terminal tangent space $\Ti$ is a real vector space;

\item\label{i:inductive-homogeneity}%
$\LL_i$ is homogeneous;

\item\label{i:inductive-continuity}%
$\LL_i$ is continuous; and

\item\label{i:angle-preserved}%
for all $U \in \Cmu$ and all $V \in \Tmu$,
$$%$$ $
  \bigl\<\LL_i(U), \LL_i(V)\bigr\> = \<U,V\>.
$$%$$ $
\end{enumerate}
\end{prop}
\begin{proof}
Claim~\ref{i:hypotheses} is by Propositions~\ref{p:conical-metric}
and~\ref{p:dévissage-initial-step}.  Claim~\ref{i:codimension} is
Lemma~\ref{l:codim-decrease}.
Claim~\ref{i:isometry} follows from Claim~\ref{i:angle-preserved}.
Claim~\ref{i:fluctuating-cone} is by Definition~\ref{d:dévissage}.
Claim~\ref{i:inductive-homogeneity} is by Definition~\ref{d:limit-log}.
Claim~\ref{i:inductive-continuity} is by
Claim~\ref{i:inductive-homogeneity} and
Definition~\ref{d:conical-metric} of conical metric, using contraction
in~Proposition~\ref{p:limit-log-contracts}.

For Claim~\ref{i:angle-preserved}, let $V$ be any element in $\Tmu$.
Given the resolving direction~$Z_0$ of~$\mu$,
Proposition~\ref{p:limit-log-preserve-mean} says that $\bmu_{Z_0} =
\LL_{Z_0}(\bmu)$ is the Fr\'echet mean of $\nu_1 =
(\LL_{Z_0})_\sharp\hmu$, so
\begin{equation}\label{eq:proof:de-singularity_1}%\end{equation}
  \nabla_{\!\bmu_{Z_0}}F_{\nu_0}(\LL_0(V)) \geq 0
  \text{ for all }V \in \Tmu
\end{equation}
because Fr\'echet means minimize Fr\'echet functions locally.  On the
other hand, by contraction under limit log maps in
Proposition~\ref{p:limit-log-contracts},
\begin{equation}\label{eq:proof:de-singularity_2}%\end{equation}
  \<U,V\> \leq \bigl\<\LL_{0}(U),\LL_{0}(V)\bigr\>,
\end{equation}
which can only decrease gradients.  Hence for all $U \in \Cmu$
\begin{equation}\label{eq:proof:de-singularity_3}%\end{equation}
  \nabla_{\!\bmu_{Z_0}}F_{\nu_0}\bigl(\LL_0(U)\bigr)
  \leq
  \nabla_{\!\bmu}F_\mu(U)
  =
  0,
\end{equation}
by Definition~\ref{d:fluctuating-cone} of fluctuating cone.  Combining
Eqs.~\eqref{eq:proof:de-singularity_1},~\eqref{eq:proof:de-singularity_2},
and~\eqref{eq:proof:de-singularity_3} yields that for $U \in
\Cmu$ and $V \in \Tmu$,
$$%$$ $
  \bigl\<\LL_{Z_0}(U),\LL_{Z_0}(V)\bigr\> = \<U,V\>.
$$%$$ $
Thus, Claim~\ref{i:angle-preserved} is satisfied for $i = 0$ and, by
induction, for all~$i$.
\end{proof}

\begin{cor}\label{c:collapse}
Given a localized measure~$\mu$ on a smoothly stratified metric
space~$\MM$, the composite d\'evissage $\LL_\circ$ in
Definition~\ref{d:composite-devissage} is a partial tangential
collapse of~$\mu$ with target~$\MM_\infty$.
\end{cor}
\begin{proof}
By Definition~\ref{d:collapse}, this is a direct consequence of
Proposition~\ref{p:ppt-L-v}.
\end{proof}

\begin{prop}\label{p:ppt-LL-confined}
Fix a localized measure~$\mu$ on a smoothly stratified metric
space~$\MM$ and a d\'evissage (Definition~\ref{d:dévissage}).  Write
\begin{itemize}
\item%
$\OO = \bnui = \bPi(\bnui)$, the origin of the terminal tangent vector
space~$\Ti$,
\item%
$\muL = (\bPi)_\sharp\nu_\infty$, the \emph{resolved measure}
on~$\Ti$, and
\item%
$\LL = \bPi\circ \LL_\circ$, the \emph{terminal collapse}
\end{itemize}
obtained by composing the terminal projection~$\bPi$ in
Definition~\ref{d:terminal-projection} with the composite d\'evissage
$\LL_\circ$ in Definition~\ref{d:composite-devissage}.  Then
$$%$$ $
  \muL
  =
  \LL_\sharp\hmu
  =
  (\LL \circ \log_\bmu)_\sharp \mu
$$%$$ $
is the pushforward of $\hmu = (\log_\bmu)_\sharp\mu$ under~$\LL$, and
$\OO$ is the Fr\'echet mean of~$\muL$.  In addition, for any $V \in
\Cmu$ and any~$X \in \Tmu$,
\begin{equation}\label{eq:confined}%\end{equation}
  \<X,V\>_\bmu = \bigl\<\LL(X), \LL(V)\bigr\>{}_\OO.
\end{equation}
\end{prop}
\begin{proof}
Since $T_\OO R_\infty = \Ti$ is Euclidean, the Fr\'echet
function~$F_{\muL}$ of~$\muL$ on~$T_\OO R_\infty$ is differentiable
and convex.  Furthermore, if $\bnu$ is the Fr\'echet mean of~$\muL$,
then
$$%$$ $
  F_{\muL}(x) - F_{\muL}(\bnu) = \frac 12\dd_\OO(x,\bnu)^2
$$%$$ $
for any $x \in T_{\bnu} R_\infty$.  To show that $\OO$
minimizes~$F_{\muL}$, it suffices to show that
\begin{equation}\label{eq:confined:1}%\end{equation}
  \nablaO F_{\muL}(Y) \leq 0
\end{equation}
for any tangent vector $Y \in \Ti$ to the terminal stratum.  Indeed,
as $\nablaO F_{\muL}(Y)$ is linear in~$V\hspace{-.2ex}$,
\eqref{eq:confined:1}~implies that $\nablaO F_{\muL}(Y) = 0$ for all
$Y \in \Ti$.  Thus $\OO$ is a stationary point of~$F_{\muL}$.
Convexity of~$F_{\muL}$ implies that $\OO$ is the minimizer
of~$F_{\muL}$.

To show \eqref{eq:confined:1}, observe that by
Corollary~\ref{c:nablamu(F)},
\begin{align}%\end{align}%$$
\label{eq:confined:2}
  \nablaO F_{\muL}(Y)
  =
  -\!\!\int_{T_\OO R_\infty} \<X,Y\>\muL(dX)
  &=
  -\!\!\int_{T_\OO\MM_\infty} \bigl\<\bPi(X),Y\bigr\> \nu_\infty(dX).
\\\label{eq:confined:3}
  &\leq
  -\!\!\int_{T_\OO\MM_\infty} \<X,Y\> \nu_\infty(dX)
  \text{ by Lemma~\ref{l:decrease-inner-product}}
\\[1ex]\notag
  &=
  \nablaO F_{\nu_\infty}(Y)
  \leq
  0,
\end{align}%$$
where $\leq 0$ is because $\OO$ minimizes $F_{\nu_\infty}$ and $Y \in
\Ti \subseteq \Emu$ by
Proposition~\ref{p:nicer-fluctuating-cone-after-limit-log}.

Differentiability of~$F_{\muL}$ (observed at the proof's start)
and~\eqref{eq:confined:2} yield $\nablaO F_{\muL}(V) = 0$, which in
particular implies equality in \eqref{eq:confined:3}.  Together these
two equalities yield
$$%$$ $
  \bigl\<\bPi(X),V\bigr\>
  =
  \<X,V\>
$$%$$ $
for $\nu_\infty$-almost all~$X$, and hence for all~$X$ by continuity
of~$\bPi$ and of inner products in
Lemmas~\ref{l:projection-is-continuous}
and~\ref{l:inner-product-is-continuous}.  Moreover, as $\bPi(V) =
V\hspace{-.2ex}$, this can be rewritten
$$%$$ $
  \bigl\<\bPi(X),\bPi(V)\bigr\>
  =
  \<X,V\>
$$%$$ $
for all~$X$.  {}From here, \eqref{eq:confined} is a consequence of
Corollary~\ref{c:collapse}.
\end{proof}

Recall the summary of definitions regarding measures and spaces from
the paragraph before Section~\ref{b:dévissage} and additional
definitions concerning d\'evissage in Definitions~\ref{d:dévissage},
\ref{d:step-devissage}, and~\ref{d:composite-devissage}, and terminal
projection in~\ref{d:terminal-projection}.

\begin{thm}\label{t:collapse}
Fix a smoothly stratified metric space~$\MM$ and on it a localized
probability measure~$\mu$ with Fr\'echet mean~$\bmu$.  The terminal
collapse $\LL = \bPi \circ \LL_\circ: \Tmu \to\nolinebreak \Ti$ for
any d\'evissage is a tangential collapse of~$\mu$ as in
Definition~\ref{d:collapse}.
\end{thm}
\begin{proof}
After Corollary~\ref{c:collapse} and
Proposition~\ref{p:ppt-LL-confined}, the only things left to prove are
homogeneity and continuity of~$\LL$, injectivity on the closure
of~$\Cmu$, and that any given d\'evissage process automatically
terminates at some finite step~$k$.

For homogeneity, since $\LL_\circ$ is homogeneous by
Proposition~\ref{p:ppt-L-v}.\ref{i:inductive-homogeneity} and
Definition~\ref{d:composite-devissage}, it suffices for the terminal
projection~$\bPi$ to be homogeneous.  That follows from homogeneity of
the concial metrics involved, given that the target~$\Ti$ of the
convex projection~$\bPi$ in Definition~\ref{d:terminal-projection} is
a (vector space and hence~a) cone.

For continuity, $\LL_\circ$ is continuous because
Proposition~\ref{p:ppt-L-v}.\ref{i:inductive-continuity} asserts that
each $\LL_i$ is continuous, and $\bPi$ is continuous by
Lemma~\ref{l:projection-is-continuous}.

That $\LL$ is injective on $\Cmu$ itself is thanks to
Corollary~\ref{c:collapse} and Proposition~\ref{p:ppt-LL-confined},
which show that $\LL$ is an isometry from~$\Cmu$ to its image.
Injectivity on the closure of~$\Cmu$ then follows from continuity of
inner products in Lemma~\ref{l:inner-product-is-continuous} and
invariance of~$\LL$ under inner products with~$\Cmu$ in
Definition~\ref{d:collapse}.\ref{i:partial-isometry}: if $Y \neq Z$
lie in the closure of~$\Cmu$, then take the inner product of~$Y$ with
any sequence of vectors in~$\Cmu$ that converges to~$Z$ to conclude
that $\LL(Y) \neq \LL(Z)$.

For termination, use Lemma~\ref{l:codim-decrease}: singularity
codimension cannot decrease indefinitely since smoothly stratified
spaces have finitely many strata by
Definition~\ref{d:stratified-space}.
\end{proof}

\begin{cor}\label{c:dim}
The vector space~$\RR^m$ in any terminal collapse $\LL: \Tmu \to
\RR^m$ is the tangent space to a particular smooth stratum
containing~$\bmu$ and hence $m \leq \dim(\MM)$.
\end{cor}
\begin{proof}
The target vector space of $\LL: \Tmu \to \Ti$ in
Theorem~\ref{t:collapse} is the terminal tangent vector space in
Definition~\ref{d:dévissage}.
\end{proof}

\begin{remark}\label{r:hull-muL}
Technically, Definition~\ref{d:hull} sets $\hull\muL$ equal to a
subset of the tangent space of~$\RR^m$.  Let us abuse notation and
identify $\hull\muL$ with a subset of~$\RR^m$~\mbox{itself}.
\end{remark}

\begin{remark}\label{r:not-proper}
It would be convenient, for various purposes, if~$\LL$ a proper map.
For example, it would imply that $\supp\muL = \LL(\supp\hmu)$.
However, $\LL$ involves convex projection~$\bPi$ in
Theorem~\ref{t:collapse}, so properness is often violated, such as
when components of~$\MM$ orthogonal to the strata meeting~$\Cmu$ are
crushed to the origin.  Nonetheless, although $\supp\muL$ need not
equal~$\LL(\supp\hmu)$, their hulls agree.
\end{remark}

\begin{lemma}\label{l:supp-NLmu}
If $\LL: \Tmu \to \RR^m$ is a tangential collapse of a measure~$\mu$
and $\muL = \LL_\sharp\hmu = (\LL \circ \log_\bmu)_\sharp \mu$, then
$\hull\muL = \hull\LL(\supp\hmu) \cong \RR^\ell$ is a linear
subspace~of\/~$\RR^m$.
\end{lemma} 
\begin{proof}
As $\hull\muL$ is a convex cone in~$\RR^m$ to begin with by
Definition~\ref{d:hull}, all that's needed for $\hull\muL \cong
\RR^\ell$ is $-V \!\in \hull\muL$ for all $V \!\in \hull\muL$.
Suppose the converse, namely $-V \not\in\nolinebreak \hull\muL$ for
some $V \!\in \hull\muL$.  Apply the hyperplane separation theorem
\cite[Theorem~II.1.6]{barvinok2002} to the two convex cones that are
$\hull\muL$ and the ray generated by~$-V$ to produce a hyperplane
$H_V$ in~$\RR^m$ that separates $\hull\muL$ from~$-V$.  Note that
$\hull\muL$ is contained in the half-space in~$\RR^m$ determined
by~$H_V$ and containing~$V$.  Because $\muL$ has Fr\'echet mean~$0$ by
Definition~\ref{d:collapse}, the support of~$\muL$ is contained
in~$H_V$, which is a convex subset of~$\RR^m$.  However, this
contradicts $V \in \hull\muL$.

Given that $\hull\muL = \RR^\ell$, the equality $\hull\muL =
\hull\LL(\supp\hmu)$ is a simple general observation, recorded in
Lemma~\ref{l:hull=subspace}, about how to generate vector spaces as
cones, since the support of the pushforward~$\muL$ is the closure of
the image of~$\supp\hmu$.
\end{proof}

\begin{lemma}\label{l:hull=subspace}
Let $\cS \subseteq \RR^\ell$ be a subset with closure~$\cS'\!$.
If\/~$\hull\cS' = \RR^\ell$ then $\hull\cS = \RR^\ell$.
\end{lemma}
\begin{proof}
Fix~$\ell$ affinely independent vectors $\ee_0,\dots,\ee_\ell$
that generate $\RR^\ell$ as a cone, so $\hull(\ee_0,\dots,\ee_\ell) =
\RR^\ell$.
% This just means that the vectors contain the origin~$\0$ in the
% interior of their convex hull;
% e.g., the standard basis vectors $\ee_1,\dots,\ee_\ell$ and $\ee_0 =
% -\sum_{i=1}^\ell e_i$ suffice
By hypothesis, $\ee_i \in \hull\cS'$ for all~$i$, so each~$\ee_i$ is a
convex combination of (finitely many) vectors in~$\cS'$.  Each of the
vectors in each of the $\ell + 1$ convex combinations is a limit of
elements of~$\cS$.  The conclusion therefore holds because the
condition $\hull(\ee_0,\dots,\ee_\ell) = \RR^\ell$ is open on
$\ee_0,\dots,\ee_\ell$, meaning that any perturbation of these vectors
still yields $\hull = \RR^\ell$.
\end{proof}

% \vfill
% \eject

%\bibliography{bibstrataCLT.bib}

\begin{thebibliography}{HMMN15}%%%%%%%%%%%%%%%%%%%%%%%%%%%%%%%%%%%%%%%
\raggedbottom%%%%%%%%%%%%%%%%%%%%%%%%%%%%%%%%%%%%%%%%%%%%%%%%%%%%%%%%%

\bibitem[ALMP12]{ALMP12}
Pierre Albin, \'Eric Leichtnam, Rafe Mazzeo, and Paolo Piazza,
\newblock \emph{The signature package on Witt spaces},
\newblock Annales scientifiques de l'\'Ecole normale sup\'erieure~\textbf{45}
  (2012), 241--310.

\bibitem[Bar02]{barvinok2002}
Alexander Barvinok,
\newblock \emph{A course in convexity}, volume~54,
\newblock Graduate Studies in Mathematics, American Mathematical Society,
  Providence, RI, 2002.
% x+366 pp. ISBN: 0-8218-2968-8
% MR1940576 (2003j:52001)

\bibitem[BBI01]{BBI01}%burago-burago-ivanov2001
Dmitri Burago, Yuri Burago, and Sergei Ivanov,
\newblock \emph{A course in metric geometry}, volume~33,
\newblock American Mathematical Soc., 2001.

\bibitem[BH13]{bridson2013metric}
Martin~R Bridson and Andr\'e Haefliger.
\newblock \emph{Metric spaces of nonpositive curvature}, volume 319,
\newblock Springer Science \& Business Media, 2013.

\bibitem[BKMR18]{bertrand-ketterer-mondello-richard2018}
J\'er\^ome Bertrand, Christian Ketterer, Ilaria Mondello, and Thomas Richard,
\newblock \emph{Stratified spaces and synthetic Ricci curvature bounds},
\newblock preprint.  \textsf{arXiv:1804.08870}

\bibitem[BL18]{barden-le2018}
Dennis Barden and Huiling Le,
\newblock \emph{The logarithm map, its limits and Fr\'echet means in orthant spaces},
\newblock Proceedings of the Londong Mathematical Society (3)
  \textbf{117} (2018), no.\,4, 751--789.

\bibitem[BP03]{bhattacharya-patrangenaru2003}
Rabi Bhattacharya and Vic Patrangenaru,
\newblock \emph{Large sample theory of intrinsic and extrinsic sample
  means on manifolds: I},
\newblock Annals of Statistics~\textbf{31} (2003), no.\,1, 1--29.

\bibitem[BP05]{bhattacharya-patrangenaru2005}
Rabi Bhattacharya and Vic Patrangenaru,
\newblock \emph{Large sample theory of intrinsic and extrinsic sample
  means on manifolds: II},
\newblock Annals of Statistics~\textbf{33} (2005), no.\,3, 1225--1259.

\bibitem[Eis95]{Eis95}
David Eisenbud,
\newblock \emph{Commutative algebra, with a view toward algebraic geometry}, vol.\,150,
\newblock Graduate Texts in Math, Springer-Verlag, New York, 1995.
% pp. xvi+785.
\newblock doi:\href{http://dx.doi.org/10.1007/978-1-4612-5350-1}%
                   {10.1007/978-1-4612-5350-1}
%         \href{http://dx.doi.org/10.1007/978-1-4612-5350-1}{978-1-4612-5350-1}
% \newblock doi:\href{http://dx.doi.org/10.1007/978-1-4612-5350-1}%
% 	  {10.1007/978-1-4612-5350-1}

\bibitem[GM88]{goresky-macpherson1988}
Mark Goresky and Robert MacPherson,
\newblock \emph{Stratified Morse theory}, volume 14,
\newblock Ergebnisse der Mathematik und ihrer Grenzgebiete (3) [Results in
  Mathematics and Related Areas (3)], Springer-Verlag, 1988.

\bibitem[HHL$^+$13]{hotz-et-al.2013}
Thomas Hotz, Stephan Huckemann, Huiling Le, J.S.\,Marron,
  Jonathan\,C.\,Mattingly, Ezra Miller, James Nolen, Megan Owen, Vic
  Patrangenaru, and Sean Skwerer,
\newblock \emph{Sticky central limit theorems on open books},
\newblock Annals of Applied Probability~\textbf{23} (2013), no.\,6, 2238--2258.

\bibitem[HMMN15]{kale-2015}
Stephan Huckemann, Jonathan Mattingly, Ezra Miller, and James Nolen,
\newblock \emph{Sticky central limit theorems at isolated hyperbolic
  planar singularities},
\newblock Electronic Journal of Probability~\textbf{20} (2015), 1--34.

% \bibitem[Kle99]{kleiner1999}
% Bruce Kleiner,
% \newblock \emph{The local structure of length spaces with curvature
%   bounded above},
% \newblock Mathematische Zeitschrift \textbf{231} (1999), no.\,3, 409--456.
% % https://mathscinet-ams-org.proxy.lib.duke.edu/mathscinet/article?mr=1704987
% % 53C23 - Global geometric and topological methods (à la Gromov);
% %         differential geometric analysis on metric spaces
% % 57M50 - General geometric structures on low-dimensional manifolds
% % \cite[p.\,412]{kleiner1999}
% % That's p.4 out of 48 in the PDF file
% % It is interesting to compare Alexandrov spaces with curvature
% % bounded above with Alexandrov spaces with curvature bounded below
% % (CBB spaces).  The papers [BGP92,Per94] show that a CBB space X has
% % very restricted structure provided its Hausdorff or topological
% % dimension is finite: [Per94] shows in particular that X is locally
% % homeomorphic to its tangent cone at each point, and that X is a
% % stratified manifold.  This implies that the links of a polyhedron with 
% % a CBB metric are homotopy equivalent to spaces with curvature ≥ 1;
% % and this is a very strong restriction on the polyhedron.  In contrast
% % to this, Berestovskii's result [Ber83] shows that an upper curvature
% % bound does not impose any restriction on topology, at least if one
% % works in the setting of polyhedra.  Also, a CBA space with finite
% % Hausdorff dimension need not have manifold points: build an ℝ-tree
% % by completing an increasing union T1 ⊂ T2 ⊂ ... of metric simplicial
% % trees where Length(Tₖ) ≤ 1 and every branch point free segment σ ⊂
% % Tₖ has length ≤ 1/k; the resulting space has Hausdorff dimension 1.

\bibitem[Kuw14]{kuwae2014}
Kazuhiro Kuwae,
\newblock \emph{Jensen's inequality on convex spaces},
\newblock Calculus of Variations and Partial Differential Equations~\textbf{49}
  (2014), no.\,3, 1359--1378.

\bibitem[MMT23a]{shadow-geom}
Jonathan Mattingly, Ezra Miller, and Do Tran,
\newblock \emph{Shadow geometry at singular points of $\CAT(\kappa)$ spaces},
\newblock preprint, 2023.

% \bibitem[MMT23b]{tangential-collapse}
% Jonathan Mattingly, Ezra Miller, and Do Tran,
% \newblock \emph{Geometry of measures on smoothly stratified metric spaces},
% \newblock preprint, 2023.

\bibitem[MMT23c]{random-tangent-fields}
Jonathan Mattingly, Ezra Miller, and Do Tran,
\newblock \emph{A central limit theorem for random tangent fields on stratified spaces},
\newblock preprint, 2023.

\bibitem[MMT23d]{escape-vectors}
Jonathan Mattingly, Ezra Miller, and Do Tran,
\newblock \emph{Central limit theorems for Fr\'echet means on stratified spaces},
% of curvature bounded above,
\newblock preprint, 2023.

\bibitem[MOP15]{centroids}
Ezra Miller, Megan Owen, and Scott Provan,
\newblock \emph{Polyhedral computational geometry for averaging metric
  phylogenetic trees},
\newblock Advances in Applied Math.\ \textbf{15} (2015), 51--91.
\newblock doi: \href{http://dx.doi.org/10.1016/j.aam.2015.04.002}%
                    {10.1016/j.aam.2015.04.002}
% \textsf{arXiv:math.MG/1211.7046}

\bibitem[Per94]{perelman1994}
G.\,Ya.\,Perel'man,
\newblock \emph{Elements of Morse theory on Aleksandrov spaces},
\newblock Algebra i Analiz \textbf{5} (1993), no.\,1, 232--241;
\newblock translation in St. Petersburg Math. J.~\textbf{5} (1994), no.\,1, 205--213.
% https://mathscinet-ams-org.proxy.lib.duke.edu/mathscinet/article?mr=1220498
% This paper proves that Aleksandrov spaces of curvature bounded
% *below* are canonically stratified.  While that means spaces whose
% curvature is bounded from both sides automatically possess such
% stratifications, \cite[(C1), p.\,6]{burago-gromov-perelman1992}
% says that geodesics in spaces of curvature bounded below do not
% bifurcate -- that is, extensions of geodesics are unique if they
% exist -- so these spaces aren't of primary interest here.

\bibitem[PL20]{pennec-lorenzi2020}
Xavier Pennec and Marco Lorenzi,
\newblock \emph{Beyond Riemannian geometry: The affine connection
  setting for transformation groups},
\newblock in \cite{pennec-sommer-fletcher2020}, Chapter 5, p.169--229.
% Editor(s): Xavier Pennec, Stefan Sommer, Tom Fletcher,
% Riemannian Geometric Statistics in Medical Image Analysis,
% Academic Press,
% 2020,
% Pages 169-229,
% ISBN 9780128147252,
% https://doi.org/10.1016/B978-0-12-814725-2.00012-1.
% (https://www.sciencedirect.com/science/article/pii/B9780128147252000121)
% \newblock doi: \href{https://doi.org/10.1016/B978-0-12-814725-2.00012-1}%
%                     {10.1016/B978-0-12-814725-2.00012-1}

\bibitem[Pfl01]{pflaum2001}
Markus~J Pflaum,
\newblock \emph{Analytic and geometric study of stratified spaces:
  contributions to analytic and geometric aspects}, volume~1768,
\newblock Springer Science \& Business Media, 2001.

\bibitem[PSF20]{pennec-sommer-fletcher2020}
Xavier Pennec, Stefan Sommer, and Tom Fletcher (eds.),
\newblock \emph{Riemannian geometric statistics in medical image analysis},
\newblock Academic Press, 2020.
% Editor(s): Xavier Pennec, Stefan Sommer, Tom Fletcher,
% Riemannian Geometric Statistics in Medical Image Analysis,
% Academic Press,
% 2020,
% Pages 169-229,
% ISBN 9780128147252,
\newblock doi: \href{https://doi.org/10.1016/B978-0-12-814725-2.00012-1}%
                    {10.1016/B978-0-12-814725-}
               \href{https://doi.org/10.1016/B978-0-12-814725-2.00012-1}%
                    {2.00012-1}
% (https://www.sciencedirect.com/science/article/pii/B9780128147252000121)

\bibitem[Shi97]{shiota97}
Masahiro Shiota,
\newblock \emph{Geometry of Subanalytic and Semialgebraic Sets},
\newblock Progress in Mathematics, vol.~150, Springer, New York, 1997.
\newblock doi:
          \href{http://dx.doi.org/10.1007/978-1-4612-2008-4}{10.1007/978-1-4612-2008-4}

\bibitem[Stu03]{sturm2003}
Karl-Theodor Sturm,
\newblock \emph{Probability measures on metric spaces of nonpositive curvature},
\newblock in Heat kernels and analysis on manifolds, graphs, and metric
  spaces: lecture notes from a quarter program on heat kernels, random
  walks, and analysis on manifolds and graphs, Contemporary
  Mathematics \textbf{338} (2003), 357--390.

%vfill

%%%%%%%%%%%%%%%%%%%%%%%%%%%%%%%%%%%%%%%%%%%%%%%%%%%%%%%%%%%%%%%%%%%%%%
\end{thebibliography}
%\bibliographystyle{alpha}

%%%%%%%%%%%%%%%%%%%%%%%%%%%%%%%%%%%%%%%%%%%%%%%%%%%%%%%%%%%%%%%%%%%%%%
%%%%%%%%%%%%%%%%%%%%%%%%%%%%%%%%%%%%%%%%%%%%%%%%%
%%%%%%%%%%%%%%%%%%%%%%%%%%%%%%%%%%%%%%%%%%%%%%%%%%%%%%%%%%%%%%%%%%%%%%

%%%%%%%%%%%%%%%%%%%%%%%%%%%%%%%%%%%%%%%%%%%%%%%%%%%%%%%%%%%%%%%%%%%%%%
%%%%%%%%%%%%%%%%%%%%%%%%%%%%%%%%%%%%%%%%%%%%%%%%%%%%%%%%%%%%%%%%%%%%%%
\end{document}